\documentclass[review,hidelinks,onefignum,onetabnum]{siamart220329}

\usepackage{lipsum}
\usepackage{amsfonts}
\usepackage{graphicx}
\usepackage{epstopdf}
\usepackage{algorithmic}
\usepackage{makecell}

\ifpdf
  \DeclareGraphicsExtensions{.eps,.pdf,.png,.jpg}
\else
  \DeclareGraphicsExtensions{.eps}
\fi

\newsiamremark{remark}{Remark}
\newsiamthm{ass}{Assumption}
\newsiamthm{lem}{Lemma}
\newsiamremark{hypothesis}{Hypothesis}
\crefname{hypothesis}{Hypothesis}{Hypotheses}
\newsiamthm{claim}{Claim}
\newsiamthm{thm}{Theorem}
\newsiamthm{prop}{Proposition}
\headers{Decentralized projected gradient descent}{W. Choi and J. Kim}

\title{On the Convergence Analysis of the Decentralized Projected Gradient Method\thanks{Submitted to the editors DATE.
\funding{This work is supported by NRF grants (No. 2016R1A5A1008055) and (No. 2021R1F1A1059671).}}}

\author{Woocheol Choi\thanks{Department of Mathematics, Sungkyunkwan University, 2066 Seobu-ro, Jangan-gu, Suwon-si, Gyeonggi-do 16419, South Korea (\email{choiwc@skku.edu})}
\and Jimyeong Kim\thanks{Stochastic Analysis and Application Research Center (SAARC), Korea Advanced Institute of Science and Technology, 291 Daehak-ro, Yuseong-gu, Daejeon 34141, South Korea (\email{jimyeongkim@kaist.ac.kr})}}

\usepackage{amsopn}

\ifpdf
\hypersetup{
  pdftitle={On the convergence analysis of the decentralized projected gradient descent},
  pdfauthor={Woocheol Choi, Jimyeong Kim \\
(Sungkyunkwan University, Mathematics)}
}
\fi

\makeatletter
\newcommand*{\rom}[1]{\expandafter\@slowromancap\romannumeral #1@}
\makeatother

\def\x{{\mathbf{x}}}

\def\1{{\mathbf{1}}}
\def\bx{{\bar{\mathbf{x}}}}

\begin{document}

\maketitle

\begin{abstract}
In this work, we are concerned with the decentralized optimization problem:
\begin{equation*}
\min_{x \in \Omega}~f(x) = \frac{1}{n} \sum_{i=1}^n f_i (x),
\end{equation*}
where $\Omega \subset \mathbb{R}^d$ is a convex domain and each $f_i : \Omega \rightarrow \mathbb{R}$ is a local cost function only known to agent $i$.   A fundamental algorithm is the decentralized projected gradient method (DPG) given by
\begin{equation*}
            x_i(t+1)=\mathcal{P}_\Omega\Big[\sum^n_{j=1}w_{ij} x_j(t) -\alpha(t)\nabla f_i(x_i(t))\Big]
\end{equation*} 
where $\mathcal{P}_{\Omega}$ is the projection operator to $\Omega$ and $ \{w_{ij}\}_{1\leq i,j \leq n}$ are communication weight among the agents. While this method has been widely used in the literature, its convergence property has not been established so far, except for the special case $\Omega  = \mathbb{R}^n$. This work establishes new convergence estimates of DPG when the aggregate cost $f$ is strongly convex and each function $f_i$ is smooth. If the stepsize is given by constant $\alpha (t) \equiv\alpha >0$ and suitably small, we prove that each $x_i (t)$ converges to an $O(\sqrt{\alpha})$-neighborhood of the optimal point. In addition, we further improve the convergence result by showing that the point $x_i (t)$ converges to an $O(\alpha)$-neighborhood of the optimal point if the domain is given the half-space $\mathbb{R}^{d-1}\times \mathbb{R}_{+}$ for any dimension $d\in \mathbb{N}$. Also, we obtain new convergence results for decreasing stepsizes. Numerical experiments are provided to support the convergence results.
\end{abstract}



\section{Introduction}\label{sec1}
Let us consider a multi-agent system with $n$ agents that form a connected network and cooperatively solve the following constrained  optimization problem:
\begin{equation}\label{problem}
\min_{x\in \Omega} f(x) =\frac{1}{n} \sum^n_{i=1} f_i(x),
\end{equation} 
where $f_i : \Omega \rightarrow \mathbb{R}$ is a local cost function only known to agent $i\in \mathcal{V}=\{1,2,\cdots,n\}$, and $\Omega\subset\mathbb{R}^d$ denotes a common convex closed set. This problem arises in many applications like engineering problems \cite{BCM, CYRC}, signal processing \cite{Boyd_Gossip, QT} and machine learning problems \cite{BCN, FCG, RB}. The works \cite{NO1, RNV} introduced the Decentralized Projected Gradient (DPG) algorithm for solving the problem \eqref{problem}, and various other algorithms have appeared for the problem \eqref{problem} such as the double averaging algorithm \cite{SBP}, the frank-wolfe type algorithm \cite{WLSM}, the primal-dual algorithm \cite{LCF}, and gradient-tracking type algorithm \cite{DMDT}. We also refer to \cite{RG} where the projected gradient method is utilized to solve the problem \eqref{problem} in a different way. 

In this work, we are concerned with the convergence property of the DPG algorithm described as follows:
\begin{equation}\label{scheme}
x_i(t+1)=\mathcal{P}_{\Omega}\bigg[\sum^n_{j=1}w_{ij}x_j(t) -\alpha(t)\nabla f_i({x}_i(t))\bigg],
\end{equation}
where $x_i (t)$ is the variable of agent $i$ at time instant $t \in \mathbb{N}\cup\{0\}$ and $\mathcal{P}_{\Omega}$ represents the projection of a vector $y$ onto the set $\Omega$, defined as:
\begin{equation}\label{projection}
\mathcal{P}_{\Omega}[y] = \arg\min_{x\in \Omega}\|x-y\|.
\end{equation}
The nonnegative weight scalars $w_{ij}$ are related to the communication pattern among agents in \eqref{problem} (see Section \ref{sec2} for the details). This algorithm is studied for various settings containing stochastic distributed optimization \cite{ALBR, KLBJ, NO3, RNV}, event-triggered communication \cite{KHT, LLSX}, and online problems \cite{AGL, CB, HCM}. In addition, the algorithm has been widely used as a backbone for developing various algorithms such as the decentralized TD learning for the multi-agent reinforcement learning \cite{DMR} and distributed model predictive control \cite{JLYC}. Despite its wide applicability, the convergence properties of the algorithm \eqref{scheme} have not been established well due to the difficulty of handling the projection operator $P_{\Omega}$ in the convergence analysis. To explain the difficulty in the convergence analysis of \eqref{scheme} compared to the case $\Omega = \mathbb{R}^d$, we note that averaging \eqref{scheme} without projection operator gives 
\begin{equation}\label{centralized}
\bar{x}(t+1) = \bar{x}(t) - \frac{\alpha (t)}{n} \sum_{i=1}^n \nabla f_i (x_i (t)),
\end{equation}
where $\bar{x}(t) = \frac{1}{n} \sum_{i=1}^n x_i (t)$. Then, if the stepsize is set to $\alpha \leq \frac{2}{\mu +L}$, one can obtain the following inequality:
\begin{equation}\label{eq-1-4}
\begin{split}
\nonumber\|\bar{x}(t+1)-x_*\|&\leq \Big(1-\frac{\mu L}{\mu+L}\alpha\Big) \|\bar{x}(t) -x_*\| + \frac{L\alpha}{n} \sum_{i=1}^n \|x_i (t) -\bar{x}(t)\|,
\end{split}
\end{equation}
 when $f$ is $\mu$-strongly convex and each $f_i$ is $L$-smooth. This inequality is a major ingredient in the convergence estimate of \eqref{scheme2} in the work \cite{YLY, CK}, but the identity \eqref{centralized} no longer holds for \eqref{scheme} due to the projection. Instead,  we proceed to obtain a sequential estimate of the quantity 
$$
\|\x(t) - \x_*\|^2 = \sum_{i=1}^n \|x_i (t) -x_*\|^2,
$$
which enables us to offset the projection operator efficiently using the contraction property of the projection operator.  As a result, we obtain a convergence result up to an error $O(\sqrt{\alpha})$. 
\subsection{Related works and contributions of this work}\label{sec1-1}
When $\Omega = \mathbb{R}^n$, the algorithm \eqref{scheme} becomes the unconstrained decentralized gradient descent as given by 
\begin{equation}\label{scheme2}
x_i(t+1)= \sum^n_{j=1}w_{ij}x_j(t) -\alpha(t)\nabla f_i({x}_i(t)).
\end{equation}
The convergence estimates of \eqref{scheme2} have been established well in the previous works \cite{NO1, YLY, CK}. In the earlier work \cite{NO1}, for a constant stepsize, Nedi\'{c}-Ozdaglar \cite{NO1} showed that the sequence converges to an $O(\alpha)$-neighborhood of the optimal set. In this work, the authors assume that the gradient $\nabla f_i$ has a uniform bound for $1\leq i \leq n$ and convexity assumption on each function $f_i$. 
In the recent work \cite{YLY}, the convergence result was established for local convex cost functions with $L$-smoothness instead of the uniform boundedness assumption on the gradient. In \cite{YLY}, the authors obtain the uniform boundedness by $L$-smoothness of a local convex function and strongly convex total cost function. They show that the sequence with constant stepsize $O(1/(\mu+L))>\alpha (t) \equiv \alpha >0$ converges to an $O(\alpha)$-neighborhood of the optimal point exponentially fast. This result was extended by the work \cite{CK} to the case with decreasing stepsize of the form $\alpha (t) = c/(t+w)^{p}$ for $0 <p \leq 1$.

As we mentioned earlier, if $\Omega \neq \mathbb{R}^n$, the convergence analysis for \eqref{scheme} becomes more challenging due to the projection operator. In the original works \cite{RNV,I, LQX}, the convergence results of \eqref{scheme} were established under the assumption that the gradient $\nabla f_i$ has a uniform bound for $1\leq i \leq n$. The work \cite{LQX} obtained the convergence rate $O(1/\sqrt{t})$ when the stepsize is chosen as $\alpha (t) = c/t$ for a suitable range of $c>0$.

The uniform boundedness of the gradient assumption was replaced by the $L$-smoothness property in the work \cite{LLSX}. For the algorithm \eqref{scheme} with $L$-smooth and convex cost functions $f_i$,  the authors showed that there exists a uniform bound  $M> 0$ of the gradients $\|\nabla f_i (x_i(t))\|$ for all $t \geq 0$ and $1 \leq i \leq n$ if the stepsize is given by a constant $\alpha >0$ satisfying $\alpha <\frac{1+\lambda_n (W)}{L}$ where $\lambda_n (W)$ denotes the smallest eigenvalue of the mixing matrix $W =\{w_{ij}\}_{1 \leq i,j \leq n}$ (refer to Section 2). For this, the DPG is interpreted as a projected gradient descent (PGD) of a cost defined on the product space $\Omega^n =  \Omega \times \cdots \Omega$ and a classical argument for the PGD was applied.  In addition, assuming also that the aggregate cost $f$ is $\mu$-strongly convex, the work \cite{LLSX} obtained the following convergence estimate:
\begin{equation}\label{eq-1-5}
\begin{split}
\|x_{i}(t+1) -x_*\| & \leq (1-\mu\alpha)^{t+1}\|\bar{x}(0) - x_*\| + \beta^{t+1} \|\x(0)\| + \frac{2\alpha M}{1-\beta} 
\\
&\quad + \sum_{s=0}^{t} (1-\mu\alpha)^{t-s}\Big( \alpha L \gamma_s + \frac{\alpha M}{\sqrt{n}}\Big),
\end{split}
\end{equation}
where $x_*$ denotes a solution of \eqref{problem} and $\gamma_s = \beta^{s} \|x_0\| + \frac{2\alpha M}{1-\beta}$. Here $\beta \in (0,1)$ is the second largest magnitude of the eigenvalues of the mixing matrix $W$ and $M>0$ denotes a uniform bound of the gradients $\|\nabla f_i (x_i (t))\|$ with respect to $t \geq 0$ and $1 \leq  i \leq n$. The above mentioned results are summarized in Table \ref{known results}.{\footnotesize
\begin{table}[ht]
\centering
\label{known results}
\begin{tabular}{|c|c|c|c|c|c| }\cline{1-6}
& {\footnotesize Cost}& {\footnotesize Smooth} & {\footnotesize Learning rate} & {\footnotesize Regret}  &{\footnotesize Rate}  
\\
\hline
&&&&&\\[-1em]
\cite{RNV} & C  & GB &\makecell{$\sum\alpha(t)=\infty$ \\ $\sum\alpha(t)^2<\infty$}& $\|x_i (t) -x_*\|$ & $o(1)$  
\\
&&&&&\\[-1em]
\hline
&&&&&\\[-1em]
\cite{I} & C  &GB &$\alpha (t) = \frac{c}{t^p}$ & \makecell{$\min_{1 \leq i \leq t}$ \\ $f(x_i (t)) - f_* $}& \makecell{$O(\frac{1}{t^{p}})$, $p \in (0, \frac{1}{2})$ \\  $O( \frac{\log t}{\sqrt{t}})$, $p = \frac{1}{2}$ \\ $O(\frac{1}{t^{1-p}})$, $p\in (\frac{1}{2},1)$}  
\\
&&&&&\\[-1em]
\hline
&&&&&\\[-0.8em]
\cite{LQX} & SC  & GB &$\alpha (t) = c/t$ & $\|x_i (t) -x_*\|$ & $O(1/\sqrt{t})$  
\\
&&&&&\\[-0.8em]
\hline
&&&&&\\[-0.8em]
\cite{LLSX} & SC  & LS &$\alpha(t)\equiv \alpha$  & $\|x_i (t) -x_*\|$ & $O(e^{-\textbf{q} t}+\alpha+ \frac{M}{\mu\sqrt{n}})$ 
\\
&&&&&\\[-0.8em]
\hline
&&&&&\\[-0.8em]
\makecell{\footnotesize Thm \\ \footnotesize \ref{thm2-3}} & SC & LS& $\alpha (t) \equiv \alpha $&$ \|x_i(t) - x_*\|$ & $O(e^{-\textbf{q} t}+\sqrt{\alpha})$  
\\ 
&&&&&\\[-1em]
\hline
&&&&&\\[-1em]
\makecell{\footnotesize Thm \\ \footnotesize \ref{thm8-1}} &\makecell{1d\\ example} & LS& $\alpha (t) \equiv \alpha $&$ \|x_i(t) - x_*\|$ & $O(e^{-\textbf{q}t} +{\alpha})$  
\\ 
&&&&&\\[-1em]
\hline
&&&&&\\[-0.8em]
\makecell{\footnotesize Thm \\ \footnotesize \ref{thm-6-10}} &\makecell{SC\\ {\footnotesize$\Omega=\mathbb{R}^{d}\times \mathbb{R}^+$}} & LS& $\alpha (t) \equiv \alpha $&$ \|x_i(t) - x_*\|$ & $O(e^{-\textbf{q}t} +{\alpha})$  
\\ 
&&&&&\\[-1em]
\hline
&&&&&\\[-0.8em]
\makecell{\footnotesize Thm \\ \footnotesize \ref{thm2-6}} & SC & LS& $\alpha (t) = \frac{v}{(t+w)^{p}}$&$ \|x_i(t) - x_*\|$ & \makecell{ $O( t^{-p/2})$, $p\in(0,1]$}   
\\ 
\hline
\end{tabular}
\vspace{0.1cm}
\caption{Convergence results for DPG. Here C and SC mean that convex and $\mu$-strongly convex, respectively and GB and LS mean that gradient bounded, i.e. $\|\nabla f_i\|_\infty <\infty$, and L-smooth, respectively. The constant $M$ means a uniform bound of a sequence $\|\nabla f_i(x_i(t))\|.$ The  $\textbf{q}>0$ is a constant determined by the properties of the cost functions and the mixing matrix.}
\end{table}
}
Here we note that the right-hand side of \eqref{eq-1-5} involves the following term
\begin{equation*}
\sum_{s=0}^{t} (1-\mu\alpha)^{t-s} \frac{\alpha M}{\sqrt{n}} =\frac{M}{\mu\sqrt{n}} \Big[ 1- (1-\mu\alpha)^{t+1}\Big],
\end{equation*}
where we used the formula $\sum^t_{s=0} r^s = (1-r^{t+1})/(1-r)$ for $r \in (0,1)$. This term converges to $\frac{M}{\mu\sqrt{n}}$ as the number of iterations $t$ goes to infinity. This limit is independent of the stepsize $\alpha >0$. Therefore, the right hand side of above estimate \eqref{eq-1-5} in the limit $t\rightarrow \infty$ involves the term $\frac{M}{\mu\sqrt{n}}$.  However, this convergence estimate is not as strong as the estimate for \eqref{scheme2} by the work in \cite{YLY} which showed that the sequence of \eqref{scheme2} with constant stepsize $\alpha (t) \equiv \alpha >0$ (below a certain threshold) converges exponentially fast to an $O(\alpha)$-neighborhood of the optimal point. Having these results, it is natural to pose the following question:

\medskip 

\noindent \emph{\textbf{Question}: Does the algorithm \eqref{scheme} with constant stepsize $\alpha >0$ converge  to a solution $x_*$ up to an $O(\alpha^{\kappa})$ error for some $\kappa>0$?}

\medskip 

\noindent It is worth mentioning that this fundamental question is still open and has not been resolved as of now. In the first main part of this work, we show that the convergence property of this question holds with $\kappa=1/2$ if the total cost function $f$ is $\mu$ strongly convex and each local cost function $f_i$ is $L_i$ smooth. 
For this, we derive two sequential inequalities with the help of the contraction property of the projection operator (see Proposition \ref{prop-3-1} and Proposition \ref{prop-3-2}). Based on these inequalities, we will also establish new convergence results of the algorithm \eqref{scheme} with decreasing stepsize. 

The second main part of this work is show that the above question holds true with $\kappa =1$ for a specific example in dimension one using the explicit formula of the algorithm \eqref{scheme}. Furthermore we verify that the property with $\kappa=1$ for a general class of functions on the half-space $\Omega =\mathbb{R}^{n-1}\times \mathbb{R}_{+}$ with assuming that the optimizer $x_*$ lies on the boundary $\partial \Omega$. For this, we develop an argument which exploits the convergence property in the normal direction to the boundary of $\Omega$ and in the perpendicular directions to the normal direction separately. We believe that this argument can be extended further to more general domains such as the smooth convex domains and $\mathbb{R}^{n-k}\times \mathbb{R}_{+}^k$ for any $1\leq k \leq n$.


\subsection{Notations and organizations}
 Before stating the results, we introduce some constants used to state and prove the results. We use $\|\cdot\|$ throughout the paper to denote the Euclidean norm. We denote $\mathbf{x}(t)$, $\bar{\mathbf{x}}(t)$ and $\mathbf{x}_*\in\mathbb{R}^{d\cdot n}$ by
\begin{equation*}
\begin{split}
\mathbf{x}(t) =[x_1(t), x_2(t)\cdots,x_n(t)]^T,\
 \bar{\mathbf{x}}(t) =[\bar{x}(t),\cdots \bar{x}(t)]^T,\
 \mathbf{x}_* = [ x_*, \cdots, x_*]^T.
 \end{split}
\end{equation*}
where $\bar{x}(t) = \frac{1}{n}\sum^{n}_{i=1} x_i(t)$. We note that
\begin{equation*}
\|\mathbf{x}(t)-\mathbf{\bar{x}}(t)\|^2=\sum^n_{i=1}\|x_i(t)-\bar{x}(t)\|^2 \ \text{and} \ \|\mathbf{x}(t)-\mathbf{x}_*\|^2=n \|\bar{x}(t)-x_*\|^2.
\end{equation*}
In addition, we use the following constants throughout the paper.

\noindent \quad $\bullet$ {$R_D^*$ is defined as $\max_{1\leq i \leq n}\| \nabla f_i(x_*)\|^2.$}

\noindent \quad {$\bullet$ $R_s$ is a uniform bound of a sequence $\|\x(t) - \x_*\|^2$ (see Theorem \ref{thm-2-11}).}

\noindent \quad $\bullet$ $\beta$ is the second largest magnitude of eigenvalues of the mixing matrix $W$ (see Assumption \ref{ass-1-1}).

\noindent \quad $\bullet$ $L>0$ is from the $L$-smoothness of  the local costs $f_i$ (see Assumption \ref{LS}).

\noindent \quad $\bullet$ $\mu >0$ is from the $\mu$-strongly convexity of the total cost $f$ (see Assumption \ref{sc}).


The rest of this paper is organized as follows. Section \ref{sec2} presents the assumptions used throughout the paper and we introduce our main result. Section \ref{sec4} contains the proofs of our main results, which are based on sequential estimates obtained in Section \ref{sec4-1}. Section \ref{sec5} provides an optimal convergence result for specific examples, such as the one-dimensional and half-space examples. Finally, in Section \ref{sec6}, we present numerical experiments to support our main theorems.

\section{Assumptions and main results}\label{sec2}
In this section, we state the assumptions on the total and local cost functions in \eqref{problem} and communication patterns among agents. We are interested in  \eqref{problem} when the local cost functions and the total cost functions satisfy the following strong convexity and smoothness assumption.
\begin{ass}\label{LS}
For each $i\in\{1,\cdots n\}$, the local cost function $f_i$ is $L$-smooth for some $L>0$, i.e., for any $x, y \in \Omega$ we have
\begin{equation*}\label{L-smooth}
\| \nabla f_i(x) - \nabla f_i(y)\| \leq L\|x-y\|\quad \forall~x,y \in \Omega.
\end{equation*}
\end{ass}
Assumption \ref{LS} implies that the total cost function $f(\cdot)=\frac{1}{n}\sum^{n}_{i=1}f_i(\cdot)$ is also $L$-smooth. 
\begin{ass}\label{sc}
 The total cost function $f(\cdot)=\frac{1}{n}\sum^{n}_{i=1}f_i(\cdot)$ is $\mu$-strongly convex for some $\mu>0$, i.e., for any $x,y\in\Omega$, we have
\begin{equation*}\label{strong}
f(y) \geq f(x) + (y-x)\nabla f(x) + \frac{\mu}{2}\|y-x\|^2.
\end{equation*}
\end{ass}
Under this assumption, the function $f$ has a unique optimizier $x_* \in \Omega$ (see Theorem 5.25 in \cite{B1}). 

In decentralized optimization, local agents share information through a communication network characterized by an undirected graph $\mathcal{G}=(\mathcal{V},\mathcal{E})$. Here, each node in $\mathcal{V}$ represents an agent, and an edge ${i,j} \in \mathcal{E}$ indicates that agent $i$ can send messages to agent $j$. The graph $\mathcal{G}$ is assumed to satisfy the following assumption. 
\begin{ass}\label{graph}
The communication graph $\mathcal{G}$ is fixed and connected. By `fixed', we mean that the graph $\mathcal{G}$ does not depend on the time $t$,  and by `connected', we mean that for any $i$ and $j$ in $\mathcal{E}$, there always exists a sequence of edges connecting them.
\end{ass}
We define the mixing matrix $W=[w_{ij}]_{1\leq i,j\leq n}$ as follows. The nonnegative weight $w_{ij}$ is given for each communication link $\{i, j\}\in \mathcal{E},$ where $w_{ij}\neq0$ if $\{i,j\}\in\mathcal{E}$ and $w_{ij} = 0$ if $\{i,j\}\notin\mathcal{E}$. In this paper, we make the following assumption on the mixing matrix $W$. 
\begin{ass}\label{ass-1-1}
The mixing matrix $W = \{w_{ij}\}_{1 \leq i,j \leq n}$ is symmetric and doubly stochastic. In addition $w_{ii}>0$. The network is connected and the weight matrix $W$ satisfies $\textrm{null}(I-W) = \textrm{span}\{1_n\}$. 
\end{ass}
Without loss of generality, we arrange the eigenvalues of $W$ to satisfy
\begin{equation*}
1=|\lambda_1 (W)| > |\lambda_2 (W)| \geq \cdots \geq |\lambda_n (W)| \geq 0,
\end{equation*}
and it is well-known that we have $\beta:= |\lambda_2 (W) | <1$ under Assumption \ref{ass-1-1} (see Section \rom{2}.B in \cite{GN}).

In this section, we present the results: consensus achievement, and convergence rate and error of the decentralized projected gradient descent scheme \eqref{scheme}. Before presenting the main results, we establish the conditions for uniform boundedness of the sequence $\{x_i (t)\}_{t \geq 0}$ in the sense that $\|\mathbf{x}(t)-\mathbf{x}_*\|^2$ is uniformly bounded for all $t\geq0$. This uniform boundedness property is crucial to obtain our main results.
\begin{thm}[Conditions for uniform bounedness]\label{thm-2-11}
There exists a constant $R_s >0$ such that
\begin{equation*}
\|\mathbf{x}(t)-\x_*\|^2 \leq R_s 
\end{equation*}
holds for all $t\geq 0$ if at least one of the following statements holds true:
\begin{enumerate}
\item $\Omega$ is bounded (no restriction on the stepsize).
\item Each local cost function $f_i$ is convex and satisfies Assumption \ref{LS}. In addition, the stepsize is constant, i.e., $\alpha (t) \equiv \alpha$, satisfying $\alpha \leq \frac{1 + \lambda_n (W)}{L}$.  
\item Each local cost function $f_i$ satisfies Assumption \ref{LS} and the total cost function $f$ satisfies Assumption \ref{sc}. In addition, the stepsize $\{\alpha (t)\}_{t \geq 0}$ is non-increasing and satisfies
$$
\alpha(t) < \min\bigg\{Z, \frac{\mu}{4c_1}, \frac{2}{L+\mu}\bigg\}.
$$ 
 Here   we have set the positive constant   $Z$ by  
$$
Z := \frac{1}{2c_3} \Big[ - \Big(c_4 + \frac{\mu}{4}\Big) + \sqrt{\Big(c_4 +\frac{\mu}{4}\Big)^2 + 4c_3 (1-\beta)}\Big]\leq1-\beta,
$$
where the constants $c_1,c_2,c_3,c_4$ are defined as follows:
\begin{equation*}\label{constants}
c_1 := \frac{3L^2 }{1-\beta},\ c_2 := \frac{3n R_D^*}{1-\beta},\ c_3 := c_1 + L^2 , \ c_4 := \frac{4L^2}{\mu}.
\end{equation*}
\end{enumerate}
\end{thm}

While Theorem \ref{thm-2-11} provides conditions for ensuring uniform boundedness, we introduce the following uniform boundedness assumption because it may hold for larger ranges of $\alpha(t)$ than that guaranteed by Theorem \ref{thm-2-11}. Proving a sharper range of $\alpha (t)$ for the uniform boundedness property would be an interesting future work.
\begin{ass}\label{ass-5}
There exists a constant $R_s>0$ such that
\begin{equation*}\label{eq-2-8}
\|\mathbf{x}(t)-{\mathbf{x}_*}\|^2 \leq R_s 
\end{equation*}
holds for all $t\geq 0$.
\end{ass}

\subsection{Consensus results}\label{sec3-1}
In centralized optimization, it suffices to demonstrate that the sequence generated by \eqref{centralized} converges to the optimal solution of \eqref{problem} because the central coordinate controls all agents simultaneously. In contrast, decentralized optimization involves each agent generating its own sequence and only sharing its information with neighboring agents. Therefore, we also need to show that each sequence generated by \eqref{scheme} converges to the same point, in which case we say the consensus is achieved. The following theorem states the consensus results based on the estimates for the consensus error $\|\mathbf{x}(t) - \mathbf{\bar{x}}(t)\|^2$.
\begin{thm}[Consensus]\label{thm2-5}
Suppose that Assumptions \ref{LS}, \ref{sc}, \ref{graph} and \ref{ass-5} hold. If a non-increasing stepsize $\{\alpha(t)\}_{t\geq0}$ satisfies $\alpha(t)\leq \frac{2}{L+\mu}$, then we have the following consensus estimates.
\begin{itemize}
\item[Case 1.] For a decreasing stepsize, we have
\begin{equation*}
\|\mathbf{x}(t)-\bar{\mathbf{x}}(t)\|^2\leq \beta^t \|\mathbf{x}(0)-\bar{\mathbf{x}}(0)\|^2 + \frac{R_c \alpha (t)^2}{(1-\beta)^2},
\end{equation*}
where{
\begin{equation}\label{rc}
R_c:=3(L^2R_s + nR_D^*)\cdot\sup_{s\geq0}\frac{\alpha(0)^2\beta^s + \alpha([s/2])^2}{\alpha(s)^2}.
\end{equation}}
\item[Case 2.] For a constant stepsize, i.e. $\alpha(t)\equiv \alpha$, we have
\begin{equation*}
\|\mathbf{x}(t)-\bar{\mathbf{x}}(t)\|^2\leq \beta^t \|\mathbf{x}(0)-\bar{\mathbf{x}}(0)\|^2 + \frac{3(L^2R_s + nR_D^* )\alpha^2}{(1-\beta)^2}.
\end{equation*}
\end{itemize}
\end{thm}
Theorem \ref{thm2-5} demonstrates that the consensus is reached exponentially fast up to an $O(\alpha (t))$ error. 
\subsection{Convergence results: Constant stepsize}\label{sec3-2}
According to Theorem \ref{thm2-5}, demonstrating the convergence of the consensus point $\bar{x}(t)$ to the optimal point $x_*$ is sufficient to establish our convergence results. The following convergence result holds when the stepsize is given by a constant. 
\begin{thm}[Convergence for constant stepsize]\label{thm2-3} 
Suppose that Assumptions \ref{LS}, \ref{sc}, \ref{graph} and \ref{ass-5} hold. If the stepsize is given by a constant  $\alpha >0$ such that $\alpha\leq\frac{2}{L+\mu}$, then
we have
\begin{equation*}
\|\bar{\mathbf{x}}(t) - \mathbf{x}_*\|^2 \leq \Big(1 - \frac{\mu \alpha}{2}\Big)^{t} \|\bar{\mathbf{x}}(0) - \mathbf{x}_*\|^2 + G_1\left(\left(1-\frac{\mu\alpha}{2}\right)^{t-1} + \beta^{\frac{t-1}{2}}\right) + G_2\alpha.
\end{equation*}
Here {the constant $G_1$ and $G_2$ are defined as
\begin{equation*}
\begin{split}
G_1 &= \frac{2\left(c_3\alpha^2 + c_4\alpha + \beta\right)R_s}{\mu\alpha} \\
G_2 &= \frac{2}{\mu}\left(\left(c_3\alpha^2 + c_4\alpha + \beta\right)\frac{3(L^2 R_s + nR_D^*)}{\left(1-\beta\right)^2} + c_1 R_s + c_2\right)
\end{split}
\end{equation*}}
where the constants {$c_1$, $c_2$, $c_3$ and $c_4$ are defined in Theorem \ref{thm-2-11}. }
\end{thm}
\begin{remark}\label{rmk-2-9}
Note that using $\sum^n_{i=1}\|a_i+b\|^2 = \sum^n_{i=1}\|a_i\|^2+\|b\|^2$ if $\sum^n_{i=1} a_i=0$ for any $a_i,b\in\mathbb{R}^d$, it follows that
\begin{equation}\label{eq-3-3}
\|\x(t) - \x_*\|^2 = \| \x(t) - \bx(t)\|^2 + \| \bx(t) - \x_*\|^2.
\end{equation} 
In view of this formula, the results of Theorem \ref{thm2-3} and Theorem \ref{thm2-5} together gives the convergence of $\|\x(t)-\x_*\|^2$. Consequently, the sequence generated by \eqref{scheme} converges exponentially fast to an $O(\sqrt{\alpha})$-neighborhood of the optimal point.
\end{remark}

We recall from \cite{YLY} that the sequence of the algorithm \eqref{scheme} on the whole space $\mathbb{R}^d$ converges to an $O(\alpha)$-neighborhood of $x_*$. This naturally leads us to pose the following question.

\medskip

\noindent \emph{\textbf{Question}: Is the convergence error $O(\sqrt{\alpha})$ in Theorem \ref{thm2-3} optimal? or can we improve the convergence error to $O(\alpha)$?}

\medskip

We give a partial answer to this question. Namely, we will establish the convergence result up to an $O(\alpha)$-neighborhood when the domain is given as the half-space  $\Omega = \{ (\tilde{x},x[d])\, |\, \tilde{x}\in \mathbb{R}^{d-1}, x[d] \geq 0\}\subseteq \mathbb{R}^d$. Here $x[k]\in\mathbb{R}$ denotes the $k$-$th$ component of the vector $x\in\mathbb{R}^d$.    

Let $x_* = (\tilde{x}_*, x_*[d])\in \mathbb{R}^{d}$ be a solution of \eqref{problem}, i.e. $x_*=\arg\min_{x\in\Omega} f(x)$. If the optimal point $x_*$ is an interior point of $\Omega$, then the projection of the DPG \eqref{scheme} becomes negligible if the point $\{x_i (t)\}_{i=1}^n$ are close enough to the point $x_*$ and $\alpha >0$ is chosen reasonably small, and so the DPG is reduced to the decentralized gradient descent \eqref{scheme2}. Thus the results of \cite{YLY, CK} on the algorithm \eqref{scheme2} can be applied to obtain the convergence result up to an $O(\alpha)$-neighborhood. Therefore, we focus only on the more intricate case that $x_*$ is on the boundary of $\Omega$. Namely, we impose the following assumption.
\begin{ass}\label{ass-5-21}
The minimizer of $f$ is on the boundary of $\Omega$, specifically, $x_* = (\tilde{x}_*, 0) $ with some $\tilde{x}_* \in \mathbb{R}^{d-1}$.   In addition each local cost function $f_i$ is $L$-smooth, and the minimizer $x_* = (\tilde{x}_*, 0)$ satisfies $\partial_d f(x_*) = \frac{1}{n}\sum^{n}_{i=1} \partial_d f_i(x_*) \geq \omega>0$. Here $\partial_kf$ denotes the $k$-$th$ component of the gradient $\nabla f$.
\end{ass}
 
\begin{thm}\label{thm-6-10}Suppose that the domain $\Omega$ is the half space $\mathbb{R}^{d-1}\times \mathbb{R}_{+}$ with any dimension $d \geq 1$ and  Assumption \ref{ass-5-21} holds. Suppose also that Assumptions \ref{LS}, \ref{sc}, \ref{graph} and \ref{ass-5} hold. If the stepsize is given by a constant  $\alpha >0$ such that $\alpha\leq\frac{2}{L+\mu}$, we have
\begin{equation}
\nonumber\underset{t \rightarrow \infty}{\textrm{limsup}}  \,\|x_k(t) - x_*\|  = O(\alpha)\quad \forall~1\leq k \leq n.   
\end{equation}
\end{thm}
 To achieve this result, we develop a new framework to handle the projection operator in an explicit way. The framework will be also useful to obtain the same result for more general domains $\Omega$ such as $\mathbb{R}^{d-k}\times \mathbb{R}_{+}^{k}$ and smooth convex domains. In \mbox{Section \ref{sec5}} we will also bring a one-dimensional for which we show that the algorithm \eqref{scheme} converges to an $O(\alpha)$ neighborhood of the optimal point by using the explicit formula of the DPG. This simple example represents the difficulty of handling the projection operator and supports the novelty of the proof of Theorem \ref{thm-6-10} obtained for general cost functions. 
 
\subsection{Convergence results: Diminishing stepsize} In this section, we consider the DPG with the decreasing stepsize of the form $v/(t+w)^p$ for some values $v>0$, $w\geq 0$ and $p \in (0,1]$, generalizing the stepsize considered in the previous works \cite{RNV, I}. The aim of this section is to obtain new convergence results of the DPG with the stepsize $\alpha (t)=v/(t+w)^p$.

 Before stating the results, we introduce {the following constants $G_3$ and $G_4$:
\begin{equation}\label{G_const}
\begin{split}
G_3 & =  \Big( c_3 \alpha (0)^2 +  c_4 \alpha (0) + \beta\Big)R_s 
\\
G_4 & = ( c_3 \alpha (0)^2 +  c_4 \alpha (0) + \beta) \frac{2R_c}{(1-\beta)^2}  + c_1 R_s +  c_2,
\end{split}
\end{equation}}
where {$c_1,\ c_2,\ c_3, \ c_4$ are defined in Theorem  \ref{thm2-3}} and $R_c$ is defined in Theorem \ref{thm2-5}.

Based on this observation, we consider a diminishing stepsize and provide the exact convergence results. Specifically, we use $\alpha(t) = v/(t+w)^p$ as the stepsize.
\begin{thm}[Convergence for diminishing stepsize]\label{thm2-6}
Suppose that Assumptions \ref{LS}, \ref{sc}, \ref{graph} and \ref{ass-5} hold.   Let $p\in(0,1]$ and assume that $\alpha(t) = \frac{v}{(t+w)^p}$ with $v, w>0$ satisfying
$$  
\alpha(0)=\frac{v}{w^p} \leq  \frac{2}{L+\mu}. 
$$
Let $\rho_1$ and $\rho_2$ be defined by 
\begin{equation}\label{rho1rho2}
{\rho_1 = \Big( \frac{w+1}{w}\Big)^{2p} ~\text{and}~ \rho_2 = \sup_{t \geq 0} \beta^{t}/\alpha (t)^2}
\end{equation}
Then we have the following results.
\begin{enumerate}
\item Assume that $p \in (0,1)$. Then we have
\begin{equation*}
  \|\bar{\mathbf{x}}(t) - \mathbf{x}_*\|^2 
\leq \frac{2e \rho_1 \Big( \rho_2 G_3 + G_4\Big)  v}{\mu} ([t/2]+w-1)^{-p} + \mathcal{R}_1(t) + \mathcal{R}_2(t),  
\end{equation*} 
where
\begin{equation*}
\begin{split}
\mathcal{R}_1(t) &= e^{-\sum^{t-1}_{s=0}\frac{\mu v}{2(s+w)^p}}\|\bar{\mathbf{x}}(0) - \mathbf{x}_*\|^2,\\
\mathcal{R}_2(t) &= \rho_1\Big( \rho_2 G_3 + G_4\Big) v^2 e^{-\frac{\mu vt}{4(t+w)^p}}\sum^{[t/2]-1}_{s=1}\frac{1}{(s+w)^{2p}}.
\end{split}
\end{equation*}
\item Assume that $p=1$ and choose $v>0$ such that $  \mu v/2 >1$. Then we have
\begin{equation}\label{eq-3-21}
 \|\bar{\mathbf{x}}(t) - \mathbf{x}_*\|^2\leq \left(\frac{w}{t+w}\right)^{ \mu v/2 }  \|\bar{\mathbf{x}}(0) - \mathbf{x}_*\|^2  + \mathcal{R}_3(t),
\end{equation}
where
\begin{equation*}
\mathcal{R}_3(t)= \frac{\rho_1}{( \mu v/2 ) -1} \Big( \frac{w+1}{w}\Big)^{ \mu v/2  } \frac{ \Big( \rho_2 G_3 + G_4\Big)  v^2}{(t+w-1)}.
\end{equation*}  
\end{enumerate}
\end{thm}
\begin{remark}
The constant $\rho_2 > 0$ defined in the above result is used in the convergence analysis. Namely, we shall obtain an intermediate inequality whose bound involves a term of the form $d_1 \beta^t + d_2 \alpha (t)^2$ for some constants $d_1 >0$ and $d_2 >0$, and we will bound the terms by $(\rho_2 d_1 + d_2)\alpha (t)^2$. 
\end{remark}
\begin{remark}\label{rmk-2-13}
Similarly as in the constant stepsize case (see Remark \ref{rmk-2-9}), the results of Theorem \ref{thm2-6} implies the convergence of $\|\x(t)-\x_*\|^2$. Specifically, for the case $p \in (0,1)$, we easily see that for any fixed $N>0$, there exists a constant $C_N>0$ independent of $t\geq0$ such that
$$
\mathcal{R}_1(t) + \mathcal{R}_2(t) \leq C_Nt^{-N}.
$$
Therefore, the DPG algorithm \eqref{scheme} converges to the optimal point with a rate of $O(t^{-p/2})$. For the case $p=1$, we observe that the convergence rate of the first term on the right-hand side of \eqref{eq-3-21} is $O(t^{-\mu v/2})$, while that of $\mathcal{R}_3(t)$ is $O(t^{-1})$. Since $\mu v/2>1$ as chosen in Theorem \ref{thm2-6}, we conclude that the sequence converges to the optimal point with a rate of $O(t^{-1/2})$.
\end{remark}

\begin{remark}
From Theorem \ref{thm2-3} and Theorem \ref{thm-6-10}, we know that the DPG with constant stepsize $\alpha >0$ converges exponentially fast to a neighborhood of the optimizer $x_*$ whose size is comparable to a power of $\alpha$. Based on this result, if we want the DPG to converge to the optimizer $x_*$ without error, then one strategy is to repeat decreasing the stepsize after suitable numbers of iterations. Optimizing this strategy and comparing its performance with the decreasing case would be interesting work.  
\end{remark}

\subsection{The main ideas of this work} 
Here we explain the main ideas of this paper, mainly for achieving the $O(\sqrt{\alpha})$-convergence result of \mbox{Theorem \ref{thm2-3}} and its improvement by Theorem \ref{thm-6-10} on the $O(\alpha)$-convergence result for the half-space. Before the explanation, we also recall the argument in the previous work \cite{LLSX} which attains an $O(1)$-convergence result. 

\subsubsection{The argument of [23]}One may see that the proof in the work \cite{LLSX} begins with writing the DPG in the following way 
\begin{equation}\label{eq-1-6}
x_k (t+1) = \sum_{j=1}^n w_{kj} x_j (t) - \alpha \nabla f_k (x_k (t)) + \phi_k (t)
\end{equation}
where $\phi_i(t)$ is the difference between DGD and DPG defined as follows:
$$
\phi_i(t) = \underbrace{{\sum^n_{j=1}w_{ij}x_j(t) -\alpha(t)\nabla f_i({x}_i(t))}}_{\text{DGD}} - \underbrace{\mathcal{P}_{\Omega}\bigg[\sum^n_{j=1}w_{ij}x_j(t) -\alpha(t)\nabla f_i({x}_i(t))\bigg]}_\text{DPG}.
$$
Averaging \eqref{eq-1-6} for $1\leq k \leq n$ one has
\begin{equation}\label{eq-1-7}
\bar{x}(t+1) =\bar{x}(t) - \frac{\alpha}{n} \sum_{i=1}^n \nabla f_k (x_i (t)) + \frac{1}{n} \sum_{i=1}^n \phi_i (t).
\end{equation}
The work \cite{LLSX} treated the last term of \eqref{eq-1-7} as an additional error term and bound it as 
\begin{equation}\label{eq-1-8}
\Big\| \frac{1}{n} \sum_{i=1}^n \phi_i (t)\Big\| =O(\alpha).
\end{equation}
Then, the contraction property of the projection operator (refer to Lemma \ref{lem-1-2}) is applied to achieve the following estimate
\begin{equation}\label{eq-1-9}
\begin{split}
&\|\bar{x}(t+1)-x_*\|
\\
&=\left\|\mathcal{P}_{\Omega}[\bar{x}(t+1)] - \mathcal{P}_{\Omega}\left[x_* - \frac{\alpha}{n} \sum_{i=1}^n \nabla f_i (x_*)\right]\right\|
\\
&\leq \Big\| \bar{x}(t+1) - \Big( x_* - \frac{\alpha}{n} \sum_{i=1}^n \nabla f_i (x_*)\Big)\Big\|
\\
& = \Big\| \bar{x}(t) - \frac{\alpha}{n} \sum_{i=1}^n \nabla f_i (x_i (t)) - \Big( x_* - \frac{\alpha}{n} \sum_{i=1}^n \nabla f_i (x_*)\Big) + \frac{1}{n} \sum_{i=1}^n \phi_i (t) \Big\|.
\end{split}
\end{equation}
Using the $L$-smooth property of the costs and the triangle inequality, one may estimate the last part of \eqref{eq-1-9} as follows:
\begin{equation}\label{eq-1-10}
\begin{split}
&\|\bar{x}(t+1)-x_*\|
\\
&\leq \Big\| \bar{x}(t) - \frac{\alpha}{n} \sum_{i=1}^n \nabla f_i (\bar{x}(t)) - \Big( x_* - \frac{\alpha}{n} \sum_{i=1}^n \nabla f_i (x_*)\Big)\Big\| + \Big\| \frac{1}{n} \sum_{i=1}^n \phi_i (t) \Big\|
\\
&\quad + \frac{L\alpha}{n} \sum_{i=1}^n \|\bar{x}(t) -x_i (t)\|.
\end{split}
\end{equation}
Now we use the strongly convexity of $f$ and \eqref{eq-1-8} to estimate the right hand side of \eqref{eq-1-10} as follows:
\begin{equation}\label{eq-1-11} 
\begin{split}
&\|\bar{x}(t+1)-x_*\|
\leq (1-c \alpha) \|\bar{x}(t)-x_*\|+ O(\alpha) + \frac{L\alpha}{n} \sum_{i=1}^n \|\bar{x}(t) -x_i (t)\|,
\end{split}
\end{equation}
provided that $\alpha >0$ is chosen less than a threshold value and $c>0$ is a constant related to the strongly convexity of $f$. We may observe that the estimate \eqref{eq-1-11} is less effective due to the fact that there is a $O(\alpha)$ term in the bound while the coefficient for the contraction of $\|\bar{x}(t)-x_*\|$ is $(1-c\alpha)$. More precisely, even if we neglect the consensus error $\frac{L\alpha}{n} \sum_{i=1}^n \|\bar{x}(t) -x_i (t)\|$ in \eqref{eq-1-11} and consider a positive sequence $\{z_t\}_{t \geq 0}$ satisfying for $t\geq 0$:
\begin{equation}\label{eq-1-12}
z_{t+1} \leq (1-c\alpha) z_t + C\alpha
\end{equation}
for some $c >0$ and $C>0$, then the tight bound of $z_t$ using \eqref{eq-1-12} is given as
\begin{equation}
\begin{split}
\nonumber z_{t} &\leq (1-c\alpha)^t z_0 + \sum_{s=0}^{t-1} (1-c \alpha)^s C\alpha
=(1-c\alpha)^t z_0 + \frac{C}{c}\Big( 1- (1-c\alpha)^t\Big).
\end{split}
\end{equation}
Taking the limit $t \rightarrow \infty$ in the above estimate gives $\lim_{t\rightarrow \infty} z_t \leq \frac{C}{c}$, which is independent of $\alpha >0$.

\subsubsection{The idea for the $O(\sqrt{\alpha})$-convergence} To obtain an improvement beyond the $O(1)$-convergence result of \cite{LLSX}, one needs a new way to estimate the sequence of the PDG involving the projection operator not using the absolute error bound \eqref{eq-1-8} for the difference $\phi_i (t)$ related to the projection. 
 
As a key part for proving the $O(\sqrt{\alpha})$-convergence result of Theorem \ref{thm2-3}, we shall use the contraction property of the projection operator (refer to Lemma \ref{lem-1-2}) in the following way:
\begin{equation}\label{eq-3-16}
\begin{split}
&\| x_i(t+1) - x_* \|^2
\\
&= \bigg\| \mathcal{P}_\Omega\bigg[\sum^n_{j=1}w_{ij} x_j(t) -\alpha(t)\nabla f_i({x}_i(t))\bigg] - \mathcal{P}_\Omega\left[ x_* -\alpha(t)\nabla f(x_*)\right]\bigg\|^2 \\
&\leq \bigg\|\sum^n_{j=1}w_{ij} x_j(t) -x_* -\alpha(t)\left(\nabla f_i({x}_i(t))-\nabla f(x_*)\right)\bigg\|^2.
\end{split}
\end{equation}
We will then sum this inequality over $1\leq i \leq n$ and derive a gain from the right hand side, using the strong convexity and the smoothness property of the costs as well as the property of the mixing matrix $W$ and some manipulations relying on the identity
$\sum^n_{i=1}\|a + b_i\|^2 = n\|a\|^2 + \sum^n_{i=1}b_i$ which holds when $\sum^n_{i=1}b_i =0$. Then, roughly we will get the following type of inequality 
\begin{equation}\label{eq-3-17}
\| \mathbf{x} (t+1) - \mathbf{x}_* \|^2 \leq  (1-c\alpha) \| \mathbf{x}
(t) - \mathbf{x}_* \|^2 + C\alpha^2
\end{equation}
for some constants $c>0$ and $C>0$. Using this estimate recursively, one has
\begin{equation}\label{eq-3-12}
\begin{split}
\| \mathbf{x} (t) - \mathbf{x}_* \|^2 &\leq  (1-c\alpha)^t \| \mathbf{x}
(0) - \mathbf{x}_* \|^2 + C\alpha^2 \sum_{s=0}^{t-1} (1-c\alpha)^{s}
\\
&\leq (1-c\alpha)^t \| \mathbf{x}
(0) - \mathbf{x}_* \|^2 + \frac{C\alpha}{c},
\end{split}
\end{equation}
which justifies the $O(\sqrt{\alpha})$-convergence of the sequence $\mathbf{x}(t+1)$ towards the optimal point $\mathbf{x}_*$.

\subsubsection{The idea for the $O({\alpha})$-convergence} Before introducing the idea on the $O(\alpha)$-convergence result behind Theorem \ref{thm-6-10}, let us try to describe briefly a reason why the above argument for Theorem \ref{thm2-3} has a limitation in comparison with the arguments used in \cite{YLY, CK} which achieves the $O(\alpha)$-convergence result of the decentralized gradient descent for the case $\Omega = \mathbb{R}^d$. Following arguments \cite{YLY, CK}, it turns out that the DGD enjoys the following types of sequential estimates
\begin{equation}\label{eq-3-13}
 \begin{split}
 \| \bar{\mathbf{x}} (t+1) - \mathbf{x}_* \|^2& \leq (1-c\alpha)  \| \bar{\mathbf{x}} (t) - \mathbf{x}_* \|^2 + C_1\alpha^3
 \\
  \| {\mathbf{x}} (t+1) - \bar{\mathbf{x}} (t+1)  \|^2& \leq \gamma   \| {\mathbf{x}} (t+1) - \bar{\mathbf{x}} (t+1)  \|^2 + C_1\alpha^2.
  \end{split}
\end{equation}
Here $c>0$, $C_1 >0$ and $C_2 >0$ are suitable constants and $\gamma <1$ is a value related to the mixing matrix $W$. Then, similarly to \eqref{eq-3-12}, we may deduce from each of these two estimates the following estimates
\begin{equation}
\begin{split}
\nonumber\| \bar{\mathbf{x}}  (t) - \mathbf{x}_* \|^2 &\leq (1-c\alpha)^t \| \bar{\mathbf{x}}(0)- \mathbf{x}_* \|^2 + \frac{C_1\alpha^2}{c}
\\
\|  {\mathbf{x}} (t) -\bar{\mathbf{x}}  (t) \|^2 &\leq \beta^t \|  {\mathbf{x}} (0) -\bar{\mathbf{x}}  (0)  \|^2 + \frac{C_2\alpha^2}{1-\beta}.
\end{split}
\end{equation}
Combining these two estimates then gives the $O(\alpha)$-convergence result of the DGD. In contrast to this argument using the two estimates of \eqref{eq-3-13} separately, the argument of subsection 3.4.2 is based on the sequential estimate \eqref{eq-3-17} which can be interpreted as a sum of the two estimates of \eqref{eq-3-13}. This makes a limitation to obtain sharp estimates for $\| \bar{\mathbf{x}}  (t) - \mathbf{x}_* \|^2$ and $\|  {\mathbf{x}} (t) -\bar{\mathbf{x}}  (t) \|^2$. 
This limitation seems difficult to overcome if the proof begins with the estimate \eqref{eq-3-16} to deal with the projection operator. 

To get around the mentioned difficulty, we develop a new argument which analyzes the convergence of the sequence $x_k (t)$ after separating the coordinates into two parts, namely the one which is influenced by the projection operator and the other one for which the projection is ignored. We refer to Subsection \ref{sec5-2} for the details. Although we limit our case to the domain $\Omega = \mathbb{R}^{d-1} \times \mathbb{R}_{+}$ in the current work,  we believe that the main idea of the argument can be extended to more general domains such as $\mathbb{R}^{q}\times \mathbb{R}_{+}^{d-q}$ and any smooth convex domain for any $0\leq q \leq d$.

\section{Analysis for main results}\label{sec4}
In this section, we present the analysis for our main results in Section 3. As discussed in Section \ref{sec1}, the projection operator makes it difficult to average the equation \eqref{scheme} to obtain \eqref{centralized} for the convergence analysis as in the unconstrained case. To get around this difficulty, we focus on estimating the quantity $\|\mathbf{x}(t+1)-\mathbf{x}_*\|^2$ instead of $\|\bar{\mathbf{x}}(t+1)-\mathbf{x}_*\|$ to analyze the sequence generated by \eqref{scheme}. More precisely, we derive sequential estimates of $\|\mathbf{x}(t+1)-\bar{\mathbf{x}}(t+1)\|^2$ and $\|\mathbf{x}(t+1)-\mathbf{x}_*\|^2$  in terms of $\|\mathbf{x}(t)-\bar{\mathbf{x}}(t)\|^2$ and $\|\bar{\mathbf{x}}(t)-\mathbf{x}_*\|^2$ with an help of the contraction property of the projection operator in Subsection \ref{sec4-1}. These estimates correspond to the results of Proposition \ref{prop-3-1} and Proposition \ref{prop-3-1} below. These results will be the main ingredients of several main results of this paper as described in \mbox{Figure \ref{main_flow}. }
\begin{figure}[!htbp]
\begin{center}
\includegraphics[width=13cm]{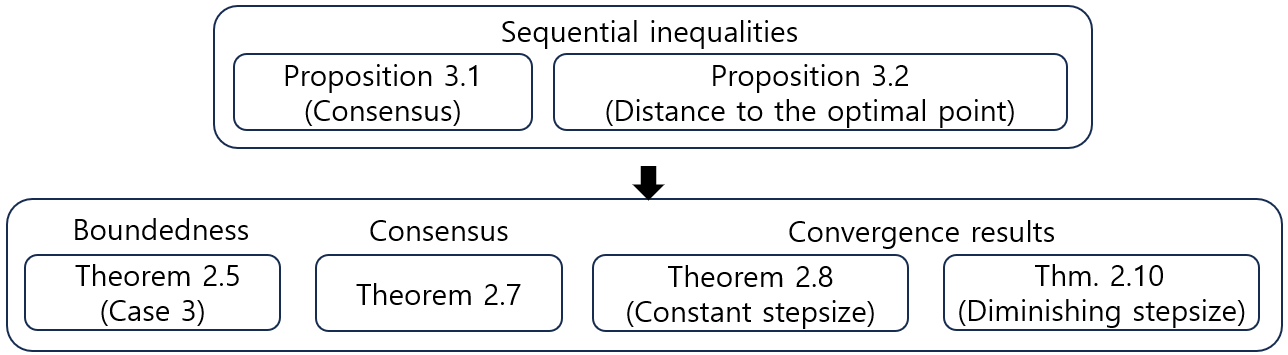}
\end{center}
\caption{The consequences of Proposition \ref{prop-3-1} and Proposition \ref{prop-3-2}.}
\label{main_flow}
\end{figure}


\subsection{Preparation: Sequential estimate}\label{sec4-1}
In this subsection, we establish sequential estimates for $\|\mathbf{x}(t+1)-\mathbf{x}_*\|^2$ and $\|\mathbf{x}(t+1)-\bar{\mathbf{x}}(t+1)\|^2$. First, we obtain a sequential estimate regarding the consensus error.
\begin{prop}\label{prop-3-1}
Suppose that Assumptions \ref{LS}-\ref{ass-1-1} hold.  If a stepsize $\{\alpha(t)\}_{t\geq0}$ satisfies $\alpha(t)\leq \frac{2}{L+\mu}$, then the sequence $\{x_i(t)\}_{t\geq0}$ generated by \eqref{scheme} satisfies the following inequality: 
\begin{equation}\label{eq-4-11}
\begin{split}
&\| \mathbf{x}(t+1) -\bar{\mathbf{x}}(t+1)\|^2\\
&\leq (c_1\alpha(t)^2 + \beta)\| \mathbf{x}(t) - \bar{\mathbf{x}}(t)\|^2 + c_1\alpha(t)^2 \|\bar{\mathbf{x}}(t)-\mathbf{x}_*\|^2 + c_2\alpha(t)^2,
\end{split}
\end{equation}
where {the constants $c_1$ and $c_2$ are defined in Theorem \ref{thm-2-11}. }
\end{prop}
\begin{proof}
See Appendix \ref{secA}
\end{proof}

The following result establishes a bound of $\|\mathbf{x}(t+1) - \mathbf{x}_*\|$ in terms of  $\|\mathbf{x}(t) - \mathbf{x}_*\|$ and $ \| \mathbf{x}(t) - \bar{\mathbf{x}}(t)\|$.  
\begin{prop}\label{prop-3-2}
Suppose that Assumptions \ref{LS}-\ref{ass-1-1} hold.  If a stepsize $\{\alpha(t)\}_{t\geq0}$  satisfies $\alpha(t)\leq \frac{2}{L+\mu}$, then the sequence $\{x_i(t)\}_{t\geq0}$ generated by \eqref{scheme}  satisfes the following inequality. 
\begin{equation*}
\begin{split}
&\|\mathbf{x}(t+1) - \mathbf{x}_*\|^2\\
&\leq \left(c_3\alpha(t)^2 + c_4\alpha(t) + \beta\right) \| \mathbf{x}(t) - \bar{\mathbf{x}}(t)\|^2 \\&\qquad\qquad\qquad\qquad\qquad\qquad+\bigg(1-\frac{\mu}{2}\alpha(t) + c_1\alpha(t)^2\bigg) \|\bar{\mathbf{x}}(t)-\mathbf{x}_*\|^2 + c_2\alpha(t)^2,
\end{split}
\end{equation*}
where {the constants $c_1$, $c_2$, $c_3$ and $c_4$ are defined in Theorem \ref{thm-2-11}. }
\end{prop}
\begin{proof}
See Appendix \ref{secA}
\end{proof}
\subsection{Consensus analysis (Proof of Theorem \ref{thm2-5})}\label{sec4-3}
In this section, we establish the consensus result of Theorem \ref{thm2-5} by using the sequential estimates of  Propositions \ref{prop-3-1} and \ref{prop-3-2}. 
\begin{proof}[Proof of Theorem \ref{thm2-5}]
We write the estimate  \eqref{eq-4-11}  in Proposition \ref{prop-3-1} using the formula \eqref{eq-3-3} as
\begin{equation*}
\begin{split}
&\| \mathbf{x}(t+1) -\bar{\mathbf{x}}(t+1)\|^2 \\
&\leq \beta\| \mathbf{x}(t) - \bar{\mathbf{x}}(t)\|^2 + \alpha(t)^2(c_1(\| \mathbf{x}(t) - \bar{\mathbf{x}}(t)\|^2+ \|\bar{\mathbf{x}}(t)-\mathbf{x}_*\|^2) + c_2) \\
&\leq \beta \| \x(t)-\bx\|^2 +\alpha(t)^2\left(c_1\|\x(t)-\x_*\|^2 + c_2\right). 
\end{split}
\end{equation*}
 Recalling that $c_1 = 3L^2/(1-\beta)$ and $c_2 = 3nR_{d,*}/(1-\beta)$, and using the bound $\| \mathbf{x}(t) -\bar{\mathbf{x}}(t)\|^2 \leq R$ from Assumption \ref{ass-5}, the above inequality can be written as follows:
\begin{equation} 
\nonumber\| \mathbf{x}(t+1) -\bar{\mathbf{x}}(t+1)\|^2 \leq \beta\| \mathbf{x}(t) - \bar{\mathbf{x}}(t)\|^2 + \frac{3\alpha(t)^2}{1-\beta}(L^2 R_s +nR_{d}^{*})\\
\end{equation}
Using this recursively, we deduce
\begin{equation}\label{eq6-33}
\| \mathbf{x}(t+1) -\bar{\mathbf{x}}(t+1)\|^2 \leq \beta^{t+1}\|\mathbf{x}(0)-\bar{\mathbf{x}}(0)\|^2 + \frac{3}{1-\beta}(L^2 R_s +nR_{d}^*)\sum^t_{s=0}\alpha(s)^2\beta^{t-s}.
\end{equation}
We will now estimate \eqref{eq6-33} for the cases where $\alpha(t)$ is non-increasing and where $\alpha(t)$ is given by a constant $\alpha>0$, respectively. If $\alpha (t)$ is non-increasing, we have
\begin{equation}\label{eq6-4}
\begin{split}
\nonumber\sum^{t}_{s=0}\alpha(s)^2\beta^{t-s} &= \sum^{[t/2]-1}_{s=0}\alpha(s)^2\beta^{t-s} + \sum^{t}_{s=[t/2]}\alpha(s)^2\beta^{t-s}  \\
&\leq \alpha(0)^2\frac{\beta^{t/2}}{1-\beta} + \alpha([t/2])^2\frac{1}{1-\beta}.
\end{split}
\end{equation}
Inserting this inequality to \eqref{eq6-33}, we have
\begin{equation*}
\|\mathbf{x}(t+1)-\bar{\mathbf{x}}(t+1)\|^2 \leq \beta^{t+1}\|\mathbf{x}(0)-\bar{\mathbf{x}}(0)\|^2  + \frac{R_c \alpha(t)^2}{(1-\beta)^2},
\end{equation*}
{where $R_c$ is defined in \eqref{rc}.}

For a constant stepsize $\alpha(t)\equiv\alpha$, we can estimate \eqref{eq6-33} as
\begin{equation*}\label{eq6-3}
\begin{split}
\| \mathbf{x}(t+1) -\bar{\mathbf{x}}(t+1)\|^2
&\leq \beta^{t+1}\|\mathbf{x}(0)-\bar{\mathbf{x}}(0)\|^2 + \frac{3\alpha^2}{1-\beta}(L^2R_s +nR_D^*)\sum^t_{s=0}\beta^{t-s}\\
&\leq \beta^{t+1}\|\mathbf{x}(0)-\bar{\mathbf{x}}(0)\|^2 + \frac{3\alpha^2}{(1-\beta)^2}(L^2R_s +nR_D^*),
\end{split}
\end{equation*}
which completes the proof.
\end{proof}

\subsection{Convergence analysis (Proof of Theorem \ref{thm2-3} and Theorem \ref{thm2-6}}\label{sec4-4}
In this section, we prove the convergence results of Theorem \ref{thm2-3} and Theorem \ref{thm2-6}. For simplicity of the exposition, we set
\begin{equation}\label{eq5-4}
A(t) = \|\mathbf{x}(t) -\bar{\mathbf{x}}(t)\|^2,\quad  B(t) =\|\bar{\mathbf{x}}(t) - \mathbf{x}_*\|^2.   
\end{equation}
Note that using \eqref{eq-3-3} it follows that
$$
 A(t) + B(t)= \|\mathbf{x}(t) -\bar{\mathbf{x}}(t)\|^2  + \|\bar{\mathbf{x}}(t) - \mathbf{x}_*\|^2=\|\mathbf{x}(t) - \mathbf{x}_*\|^2.
$$ 
In addition, we can rewrite Theorem \ref{thm2-5} and Proposition \ref{prop-3-2} as
\begin{equation}\label{eq7-2}
A(t)\leq \beta^t A(0) + \frac{C\alpha (t)^2}{(1-\beta)^2},
\end{equation}
where $C=3(L^2 R_s  + nR_D^*)$ if $\alpha(t)\equiv \alpha$ or $C=R_D$ otherwise, and
\begin{equation}\label{eq7-3}
\begin{split}
&A(t+1)+B(t+1)\\
&\leq (c_3\alpha(t)^2 + c_4\alpha(t) + \beta) A(t) +\bigg(1-\frac{\mu}{2}\alpha(t) + c_1\alpha(t)^2\bigg) B(t) + c_2\alpha(t)^2.
\end{split}
\end{equation}
Then we further write \eqref{eq7-3} as follows:
\begin{equation*}
\begin{split}
B(t+1)&\leq A(t+1) + B(t+1) \\
&\leq \left(1-\frac{\mu}{2}\alpha(t)\right)B(t) +  \left(c_3\alpha(t)^2 + c_4\alpha(t) + {\beta}\right)A(t) + \left(c_1 B(t) + c_2\right)\alpha(t)^2 \\
&\leq \left(1-\frac{\mu}{2}\alpha(t)\right)B(t) +  (c_3\alpha(t)^2 + c_4\alpha(t) + {\beta})A(0)\beta^t \\
&\quad + \left((c_3\alpha(t)^2 + c_4\alpha(t) + \beta)\frac{C}{(1-\beta)^2} + c_1B(t) + c_2\right)\alpha(t)^2,
\end{split}
\end{equation*}
where we used  \eqref{eq7-2} in the last inequality.
Since $A(t),B(t)\leq R_s$ by Assumption \ref{ass-5}, it follows that
\begin{equation}\label{eq-7-20}
\begin{split}
B(t+1) &\leq \left(1-\frac{\mu}{2}\alpha(t)\right) B(t) + (c_3\alpha(t)^2 + c_4\alpha(t) + \beta)R_s \beta^t \\
&\quad + \left((c_3\alpha(t)^2 + c_4\alpha(t) + \beta)\frac{C}{(1-\beta)^2} + c_1 R_s + c_2\right)\alpha(t)^2.
\end{split}
\end{equation}
This inequality serves as a main tool for deriving our main results, namely Theorem \ref{thm2-3} and Theorem \ref{thm2-6}. First, we handle the case where the stepsize is constant.
\begin{proof}[Proof of Theorem \ref{thm2-3}]
Since $\alpha(t)\equiv\alpha$ and $C=3(L^2R_s +nR_D^*)$ in this case, we can write \eqref{eq-7-20} as
\begin{equation*}
\begin{split}
B(t+1) &\leq \left(1-\frac{\mu}{2}\alpha\right) B(t) + \left(c_3\alpha^2 + c_4\alpha + \beta\right)R_s \beta^t \\
&\quad + \left((c_3\alpha^2 + c_4\alpha + \beta)\frac{3(L^2R_s + nR_D^*)}{(1-\beta)^2} + c_1 R_s + c_2\right)\alpha^2.
\end{split}
\end{equation*}
To estimate the sequence $B(t)$  we consider the following two sequences $\{q_1 (t)\}_{t\geq0}$ and $\{q_2(t)\}_{t\geq0}$ satisfying $q_1(0)=0$, $q_2(0)=B(0)$ and
\begin{equation*}
\begin{split}
q_1 (t+1)& = \Big( 1- \frac{\mu}{2}\alpha \Big) q_1 (t) + \left(c_3\alpha^2 + c_4\alpha + \beta\right)R_s \beta^{t},
\\
q_2(t+1)& = \Big( 1- \frac{\mu}{2}\alpha \Big) q_2(t) +  \left(\left(c_3\alpha^2 + c_4\alpha + \beta\right)\frac{3(L^2 R_s + nR_D^*)}{\left(1-\beta\right)^2} + c_1R_s + c_2\right)\alpha^2.
\end{split}
\end{equation*}
It then easily follows that $B(t) \leq q_1 (t) + q_2 (t)$ for all $t \geq 0$. 
We first estimate the sequence $q_1(t)$. We solve the recursive formula of  $q_1(t)$ to find
\begin{equation*}
\begin{split}
q_1(t) &= \left(1-\frac{\mu}{2}\alpha\right)^{t}q_1(0) +\left(c_3\alpha^2 + c_4\alpha + \beta\right)R_s \sum_{s=0}^{t-1}   \beta^{s} \Big( 1- \frac{\mu \alpha}{2}\Big)^{t-1-s}\\
&\leq t{\left(c_3\alpha^2 + c_4\alpha + \beta\right)R_s} \max\Big\{\left(1-\frac{\mu\alpha}{2}\right)^{t-1}, ~ \beta^{{t-1}}\Big\},
\end{split}
\end{equation*}
where we used $q_1(0)=0$ and the following estimates for the last inequality:
\begin{equation*}
\begin{split}
\sum_{s=0}^{t-1}   \beta^{s} \Big( 1- \frac{\mu \alpha}{2}\Big)^{t-1-s} & \leq t\max\Big\{\left(1-\frac{\mu\alpha}{2}\right)^{t-1}, ~ \beta^{{t-1}}\Big\}.
\end{split}
\end{equation*}
Next, we estimate the term $q_2(t)$ as follows:
\begin{equation*}
\begin{split}
q_2(t)  &= \Big(1 - \frac{\mu \alpha}{2}\Big)^{t} q_2(0)  \\
&\qquad+ \left(\left(c_3\alpha^2 + c_4\alpha + \beta\right)\frac{3(L^2 R_s + nR_D^*)}{\left(1-\beta\right)^2} + c_1 R_s + c_2\right)\alpha^2\sum^{t-1}_{s=0} \left(1-\frac{\mu\alpha}{2}\right)^{t-1-s}\\
&\leq \Big(1 - \frac{\mu \alpha}{2}\Big)^{t} B(0) + \frac{2\alpha}{\mu} \left(\left(c_3\alpha^2 + c_4\alpha + \beta\right)\frac{3(L^2R_s + nR_D^*)}{\left(1-\beta\right)^2} + c_1 R_s + c_2\right)
\end{split}
\end{equation*}
Here used $q_2(0)=B(0)$ and $\sum^{t-1}_{s=0} \left(1-\frac{\mu\alpha}{2}\right)^{t-1-s} \leq \frac{2}{\mu\alpha}$ for the last inequlaity.
Combining the above two estimates, we obtain
\begin{equation*}
B(t) \leq q_1(t)+q_2(t)\leq \Big(1 - \frac{\mu \alpha}{2}\Big)^{t} B(0) + t G_1 \max \Big\{ \left(1-\frac{\mu\alpha}{2}\right)^{t-1} + \beta^{{t-1}}\Big\} + G_2\alpha,
\end{equation*}
where {$G_1$ and $G_2$ are defined in Theorem \ref{thm2-3}.} This inequality is the desired estimate. 
\end{proof}
Next, we prove Theorems \ref{thm2-6}, where the stepsize is defined as $\alpha (t) = \frac{v}{(t+w)^p}$ for $0<p\leq1$. Since $\alpha(t) \leq \alpha(0)$, we can write \eqref{eq-7-20} as
\begin{equation}\label{eq-5-9}
\begin{split}
B(t+1) &\leq \left(1-\frac{\mu}{2}\alpha(t)\right) B(t) + (c_3\alpha(0)^2 + c_4\alpha(0) + \beta)R_s \beta^t \\
&\quad + \left((c_3\alpha(0)^2 + c_4\alpha(0) + \beta)\frac{R_c}{(1-\beta)^2} + c_1 R_s + c_2\right)\alpha(t)^2.
\end{split}
\end{equation}
Then, it follows from \eqref{eq-5-9} that
\begin{equation*}
B(t+1) \leq \Big( 1- \frac{\mu}{2}\alpha (t)\Big) B(t) + \Big( \rho_2 G_3+ G_4\Big) \alpha (t)^2,
\end{equation*}
where {$G_3$ and $G_4$ are defined in \eqref{G_const}, and $\rho_2$ is defined in \eqref{rho1rho2}.} To obtain a convergence rate from the above inequality, we recall the following lemma.
\begin{lem}[Proposition 5.1 in \cite{CK}]\label{lem-4-3} 
 Let $ p \in (0,1]$ and $q>0$. Take $C_1 >0$ and $w \geq 1$ such that $C_1/w^p <1.$ 
Suppose that the sequence $\{H(t)\}_{t\geq0}$ satisfies
\begin{equation*}\label{4-4}
H(t) \leq \bigg(1-\frac{C_1}{(t+w-1)^p}\bigg) H(t-1) + \frac{C_2}{(t+w-1)^{p+q}} \quad \text{for all $t\geq1$}.
\end{equation*}
Set $Q = \Big( \frac{w+1}{w}\Big)^{p+q}$. Then $H(t)$ satisfies the following bound.

\medskip 

\noindent \textbf{Case $p\in(0,1)$.} If $p \in (0,1)$, we have
\begin{equation*}
 H(t) \leq   \delta \cdot([t/2]+w-1)^{-q} + \mathcal{M}(t),
\end{equation*}
where  $\delta =  \frac{QC_2}{C_1} e^{\frac{C_1}{w^p}}$ and
\begin{equation*}
\mathcal{M}(t) =e^{-\sum^{t-1}_{s=0}\frac{C_1}{(s+w)^p}}H(0)  + QC_2e^{- \frac{C_1t}{2(t+w)^p}}  \sum_{s=1}^{[t/2]-1} \frac{1}{(s+w)^{p+q}}.
\end{equation*}
Here the second term on the right hand side is assumed to be zero for $1\leq t \leq 3$.
\medskip

\noindent \textbf{Case $p=1$.} If $p=1$, then we have
\begin{equation*}
 H(t) \leq \Big( \frac{w}{t+w}\Big)^{C_1} H(0) + \tilde{\mathcal{M}} (t),
\end{equation*}
where
\begin{equation*}
  \tilde{\mathcal{M}}    = \left\{\begin{array}{ll}  \frac{w^{C_1 -q}}{q-C_1}\cdot\frac{QC_2}{(t+w)^{C_1}}& \textrm{if}~q>C_1
\\
\log \left(\frac{t+w}{w}\right)\cdot\frac{QC_2}{(t+w)^{C_1}} &\textrm{if}~ q=C_1
 \\
\frac{1}{C_1-q}\cdot\left( \frac{w+1}{w}\right)^{C_1}\cdot \frac{QC_2}{(t+w+1)^q}& \textrm{if}~ q<C_1.
\end{array}\right.
\end{equation*} 
\end{lem}
We are ready to prove Theorem \ref{thm2-6}.
\begin{proof}[Proof of Theorem  \ref{thm2-6}]  For the stepsize $\alpha (t) = v/(t+w)^p$, the estimate \eqref{eq-7-20} reads as
\begin{equation*}
B(t+1) \leq \Big( 1- \frac{\mu v}{2(t+w)^p}\Big) B(t) + \Big( \rho_2 G_3 + G_4\Big)\frac{  v^2}{(t+w)^{2p}}.
\end{equation*}
By applying Lemma \ref{lem-4-3} for the case $p\in(0,1)$, we have
\begin{equation*}
B(t) \leq \frac{4\rho_1 \Big( \rho_2 G_3 + G_4\Big)  v}{\mu} ([t/2]+w-1)^{-p} + \mathcal{R}_1(t) + \mathcal{R}_2(t),
\end{equation*}
where 
\begin{equation*}
\begin{split}
\mathcal{R}_1(t) &= e^{-\sum^{t-1}_{s=0}\frac{\mu v}{2(s+w)^p}}B(0)\\
\mathcal{R}_2(t) &= \rho_1 \Big( \rho_2 G_3 + G_4\Big) v^2 e^{-\frac{\mu vt}{4(t+w)^p}}\sum^{[t/2]-1}_{s=1}\frac{1}{(s+w)^{2p}}
\end{split}
\end{equation*}
with constant $\rho_1 = \Big( \frac{w+1}{w}\Big)^{2p}$. It gives the desired estimate for $p \in (0,1)$.

For the stepsize $\alpha (t) = v/(t+w)$, 
the estimate \eqref{eq-7-20} is written as
\begin{equation*}
B(t+1) \leq \Big( 1-\frac{\mu v}{2 (t+w)}\Big)B(t) +\Big( \rho_2 G_3 + G_4\Big) \frac{v^2}{(t+w)^2}.
\end{equation*}
Choose $v>0$ so that $C_1 = \mu v/2 >1$. Then we use Lemma \ref{lem-4-3} for the case $p=1$ to derive the following estimate 
\begin{equation*}
B(t) \leq \Big( \frac{w}{t+w}\Big)^{C_1} B(0) + \frac{1}{C_1 -1} \Big( \frac{w+1}{w}\Big)^{C_1} \frac{\rho_1 \Big( \rho_2 G_3 + G_4\Big)  v^2}{(t+w-1)},
\end{equation*}
where {$\rho_1$ is defined in \eqref{rho1rho2}}. This is the desired estimate.
\end{proof}

\subsection{Uniform boundedness of the sequence}\label{sec4-2}
In subsections \ref{sec4-1}, \ref{sec4-3}, and \ref{sec4-4}, we derived our main results under the uniform boundedness assumption $\|\x(t)-\x_*\|^2\leq R_s$ (refer to Assumption \ref{ass-5}). This uniform boundedness assumption is satisfied if at least one of these conditions which are presented in Theorem \ref{thm-2-11} is satisfied. In this subsection, we prove the uniform boundedness of the sequence $\{x_i(t)\}_{t\geq0}$ as stated in Theorem \ref{thm-2-11}. Notably, uniform boundedness is trivial for the first case where $\Omega$ is assumed to be bounded. Next, we prove the theorem for the second case where each local cost function $f_i$ is convex and satisfies Assumption \ref{LS}. In addition, the stepsize is constant, i.e., $\alpha (t) \equiv \alpha$, satisfying $\alpha \leq \frac{1 + \lambda_n (W)}{L}$
\begin{proof}[Proof of Theorem \ref{thm-2-11} for case 2]
We first recall Consider the following functional $E_{\alpha}:(\mathbb{R}^{d})^n \rightarrow \mathbb{R}$ defined as
\begin{equation*}
E_{\alpha} (x) = \frac{1}{2} \Big( \sum_{i=1}^n \|x_i\|^2 - \sum_{i=1}^n \sum_{j=1}^n w_{ij} \langle x_i, x_j \rangle \Big) + \alpha \sum_{i=1}^n f_i (x_i).
\end{equation*}
Then 
\begin{equation*}
x(t+1) = P_{\Omega^n} \Big(x(t) -\nabla E_{\alpha}(x(t))\Big),
\end{equation*}
where $\Omega^n = \{(x_1, \cdots, x_n) \in (\mathbb{R}^d)^n~:~x_j \in \Omega ~\textrm{for}~ 1 \leq j \leq n\}$.
The function $E_{\alpha}$ is convex and smooth with constant $1-\lambda_n (W) + \alpha L$, where $\lambda_n (W)$ is the smallest eigenvalue of $W$ (refer to \cite{YLY}). Then, we may use the general result for the projected gradient descent (see e.g., \cite{B}) to conclude that
 the sequence $\{x(t)\}_{t \geq 0}$ is uniformly bounded if
\begin{equation*}
1 \leq \frac{2}{1- \lambda_n (W) +\alpha L},
\end{equation*}
which is equivalent to $\alpha \leq \frac{1+\lambda_n (W)}{L}$.
\end{proof}
Next, we consider Case 3 in Theorem \ref{thm-2-11}. In the following lemma, we find a sequential inequality for the sequence $\{\|\x(t) - \x_*\|^2\}_{t\geq0}$ and find its uniform boundedness, which also implies the uniform boundedness of $\{\|\x(t) - \bx(t)\|^2\}_{t\geq0}$ and $\{\|\bx(t)-\x_*\|^2\}_{t\geq0}$. It contains the proof of Theorem \ref{thm-2-11} for case 3.
\begin{lem}\label{lem-3-3}
Suppose that Assumptions \ref{LS}-\ref{ass-1-1} hold and the stepsize $\{\alpha(t)\}_{t\geq0}$ is non-increasing and satisfies
\begin{equation}\label{eq5-7}
\nonumber\alpha(t) \leq \min\bigg\{Z, \frac{\mu}{4c_1}, \frac{2}{L+\mu}\bigg\}
\end{equation}
where $Z$ is defined in Theorem \ref{thm-2-11}.
Then the sequence $\{x_i(t)\}_{t\in\mathbb{N}_0}$ generated by \eqref{scheme} satisfies the following statements. 
\begin{enumerate}
\item We have
\begin{equation}\label{4-11}
\|\x(t+1) - \x_*\|^2 \leq \Big(1-\frac{\mu}{4}\alpha(t)\Big)\|\x(t) - \x_*\|^2 + c_2\alpha(t)^2.
\end{equation}
\item There exists $R >0$ such that
\begin{equation*}
\|\x(t) - \x_*\|^2 \leq R, \ \text{for all $t\in\mathbb{N}$.} 
\end{equation*}
In fact, we may set $R   = \max\left\{\frac{4c_2  }{\mu}\sup_{t \geq 0} \alpha (t),~ \|\x(0) - \x_*\|^2\right\}$
\end{enumerate} 

\end{lem}
\begin{proof}
We first prove \eqref{4-11}. By Proposition \ref{prop-3-2} and using the notation \eqref{eq5-4} again, we have the following estimate
\begin{equation}\label{eq-4-10}
\begin{split}
&\|\x(t+1) - \x_*\|^2 \\
&\leq (c_3\alpha(t)^2 + c_4\alpha(t) + \beta)A(t)+\bigg(1-\frac{\mu}{2}\alpha(t) + c_1\alpha(t)^2\bigg) B(t) + c_2\alpha(t)^2.
\end{split}
\end{equation}
In view of the relation $\|\x(t)-\x_*\|^2 = A(t) + B(t)$, our objective is to bound the coefficients of $A(t)$ and $B(t)$ in the aobve inequality by $(1 - c\alpha(t))$ for a fixed value $c>0$. Choosing the value $c$ as $\mu/4$, we deduce the desired esitmate \eqref{4-11} as follows:
\begin{equation*}
\begin{split}
\|\x(t+1) - \x_*\|^2&\leq \left(1-\frac{\mu}{4}\alpha(t))\right)(A(t)+B(t)) + c_2\alpha(t)^2\\
&= \left(1-\frac{\mu}{4}\alpha(t))\right)\|\x(t+1) - \x_*\|^2 + c_2\alpha(t)^2,
\end{split}
\end{equation*}
provided that we have the following two inequalities for all $t \geq 0$:
\begin{equation}\label{eq-4-1}
1-\frac{\mu}{2}\alpha(t) + c_1\alpha(t)^2 < 1-\frac{\mu}{4}\alpha(t),
\end{equation}
\begin{equation}\label{eq-4-2}
c_3\alpha(t)^2 + c_4\alpha(t) +\beta < 1- \frac{\mu}{4}\alpha(t).
\end{equation} 
Since $\alpha(t)\leq \frac{\mu}{4c_1}$, we obtain \eqref{eq-4-1} as follows:
\begin{equation*}
1-\frac{\mu}{2}\alpha(t) + c_1\alpha(t)^2 = 1 - \left(\frac{\mu}{2}-c_1\alpha(t)\right)\alpha(t)\leq 1 - \frac{\mu}{4}\alpha(t).
\end{equation*}
To obtain \eqref{eq-4-2}, we note that \eqref{eq-4-2} is equivalent to 
\begin{equation}\label{eq5-14}
c_3 \alpha(t)^2 + \Big(c_4 + \frac{\mu}{4}\Big)\alpha(t) + \beta -1 \leq 0.
\end{equation}
Therefore, the following inequality is a sufficient condition for \eqref{eq5-14}:
\begin{equation*}
\alpha(t)  \leq Z:= \frac{1}{2c_3} \left[ - \Big(c_4 +\frac{\mu}{4}\Big) + \sqrt{\Big(c_4 +\frac{\mu}{4}\Big)^2 + 4c_3 (1-\beta)}\right].
\end{equation*}
Since $\beta<1$, we have $Z>0$.
This proves the first estimate of the lemma. 

In order to show the second estimate, we argue by induction. Fix a value $R>0$ and assume that $\|\x(t) - \x_*\|^2 \leq R$ for some $t \in \mathbb{N}_0$. Then, it follows from \eqref{4-11} that
\begin{equation}\label{eq-4-15}
\begin{split}
\|\x(t+1) - \x_*\|^2&\leq \left(1-\frac{\mu}{4}\alpha(t)\right) R +c_2\alpha(t)^2= R - \left(\frac{\mu}{4}R -c_2\alpha(t)\right)\alpha(t).
\end{split}
\end{equation}
If we set $R = \frac{4c_2 }{ \mu }\max_{t \geq 0} \alpha (t)$, then we have
$$
\frac{\mu}{4}R -c_2\alpha(t)\geq  0.
$$ 
Inserting this into \eqref{eq-4-15} proves that $\|\x(t+1) - \x_*\|^2\leq R$. Therefore we have $\|\x(t) - \x_*\|^2 \leq R$ for any $t \geq 0$.
\end{proof}
 
 \begin{remark}
 We remark from the above proof that for the uniform boundedness result of Lemma \ref{lem-3-3}, we may have a weakened version of \eqref{4-11} of the form
 \begin{equation}
\nonumber\|\x(t+1) - \x_*\|^2 \leq \Big(1-\zeta\alpha(t)\Big)\|\x(t) - \x_*\|^2 + c_2 \alpha(t)^2,\end{equation}
 where $\zeta\in \mathbb{R}$ is chosen as $\zeta  \in (0, \mu/2)$. Then the associated inequalities \eqref{eq-4-1} and \eqref{eq-4-2} are relaxed, and so the range of $\alpha (t)$ can be enlarged. 
 \end{remark}

\section{Improved convergence estimate}\label{sec5}
In this section, we provide the proof of Theorem \ref{thm-6-10} which establishes the $O(\alpha)$-convergence result for the half-space. Before this, we exhibit a simple example in dimension one for which we show the $O(\alpha)$-convergence result based on a direct computation of the algorithm. This example suggests that obtaining the $O(\alpha)$-convergence result for the DPG becomes more intricate than that of the DGD due to the projection operator.  
 
\subsection{One-dimensional example}\label{sec5-1}
We consider the functions \mbox{$g_1, g_2 : [0,\infty)\rightarrow \mathbb{R}$} defined by
\begin{equation}\label{eq8-2}
g_1 (x) = 5x^2 \quad \textrm{and}\quad g_2 (x) = -3x^2, \quad x\in [1,\infty]
\end{equation}
and the mixing matrix $\tilde{W}$ defined by
\begin{equation}\label{eq8-3}
\tilde{W} = \begin{pmatrix} 2/3 & 1/3 \\ 1/3 & 2/3 \end{pmatrix}
\end{equation}
satisfying Assumption \ref{ass-1-1}. We note that the total cost function $g = (g_1+g_2)/2$ has the optimal point at $x_*=1$. Then we can represent the projected decentralized gradient descent algorithm with a constant stepsize $\alpha$ explicitly as follows:
\begin{equation}\label{eq8-4}
\begin{split}
x_1(t+1) & =  \max \Big[\frac{2}{3}x_1(t) + \frac{1}{3}x_2(t) - 10\alpha x_1(t),~ 1\Big],
\\
x_2(t+1) & = \max \Big[ \frac{1}{3}x_1(t) + \frac{2}{3}x_2(t) +6\alpha x_2(t),~1\Big].
\end{split}
\end{equation}
We first establish that the state $(x_1(t),x_2(t))$ generated by the algorithm \eqref{eq8-4} will be confined to a certain region after a finite number of iterations. 
\begin{lem}\label{lem8-1}
Let $g_1(x)$, $g_2(x)$ and the mixing matrix $\tilde{W}$ be defined by \eqref{eq8-2} and \eqref{eq8-3} and let $\mathbf{x}(t) = (x_1(t), x_2(t))$ be the state at $t\geq0$ generated by \eqref{eq8-4}. Then for any initial state $\mathbf{x}(0) = (x_1(0), x_2(0)) \in [0,1\infty)^2$ and $\alpha \in (0,1/45)$, there exists $t_0 \leq \log_{\lambda+} (1/\|\mathbf{x}(0)\|_2)$ such that $x_1 (t_0+1) =1$ and $x_2 (t_0 +1) \leq 1 + 30\alpha$, where $\lambda_{+}\in (0,1)$ is defined as
\begin{equation*}
\lambda_{+} = \frac{2}{3} - 2\alpha + \sqrt{\frac{1}{9} +64\alpha^2},
\end{equation*}
which is  less than $1$ for $\alpha <1/45$. 
\end{lem}
\begin{proof}[Proof]
We divide the proof into three steps.

\noindent \textbf{Step 1}: There is an integer $t_0 \leq \log \|\mathbf{x}(0)\|_2/ \log(1/\lambda_{+})$ such that \mbox{$x_1 (t_0+1) =1$.}

To show this, we  assume that 
\begin{equation}\label{eq-5-4}
x_1 (s) >1 \quad \forall~0 \leq s \leq \|\mathbf{x}(0)\|_2/ \log(1/\lambda_{+})+1.
\end{equation} Then it follows from \eqref{eq8-4} that
\begin{equation}\label{eq8-9}
\begin{pmatrix} x_1 (s+1) \\ x_2 (s+1) \end{pmatrix} = \begin{pmatrix} \frac{2}{3}-10\alpha & \frac{1}{3} \\ \frac{1}{3} & \frac{2}{3} + 6\alpha \end{pmatrix} \begin{pmatrix} x_1 (s) \\ x_2 (s) \end{pmatrix},
\end{equation}
where we also used that $x_2 (s+1)>1$ since $x_1 (s) \geq 1$ and $x_2 (s) \geq 1$ for all $s \geq 0$.
The eigenvalues of the matrix on the right hand side of \eqref{eq8-9} are
\begin{equation*}
\lambda_{\pm} = \frac{2}{3} - 2\alpha \pm \sqrt{\frac{1}{9} +64\alpha^2} 
\end{equation*}
which are positive and less than $1$ for $\alpha \in (0, 1/45)$. Therefore
\begin{equation*}
\|(x_1 (s+1), x_2 (s+1))\|_2 \leq \lambda_{+} \|(x_1 (s), x_2 (s))\|_2,
\end{equation*}
and so 
\begin{equation*}
\|(x_1 (s), x_2 (s))\|_2 \leq \lambda_{+}^{s} \|(x_1 (0), x_2 (0))\|_2\quad \forall~0 \leq s \leq \|\mathbf{x}(0)\|_2/ \log(1/\lambda_{+})+1.
\end{equation*}
Let $s_0$ be the maximal integer $\leq \|\mathbf{x}(0)\|_2/ \log(1/\lambda_{+})+1$. Then it follows from the above inequality that
\begin{equation}
\nonumber\|(x_1 (s_0), x_2 (s_0))\|_2 \leq \lambda_{+}^{s_0} \|(x_1 (0), x_2 (0))\|_2 \leq 1.
\end{equation}
This yields that $x_1 (s_0) \leq 1$, which contradicts to the fact that $x_1 (s_0 ) >1$ from the assumption \eqref{eq-5-4}. Therefore \eqref{eq-5-4} is not true, and so there exists an integer $t_0 \leq \log_{\lambda+} (1/\|\mathbf{x}(0)\|_2)$ such that $x_1 (t_0 +1) \leq 1$.

\medskip 

\noindent \textbf{Step 2}: It holds that 
\begin{equation}\label{eq-5-7}
x_1 (t_0) \leq (3-x_2 (t_0))/(2-30\alpha)\quad \textrm{and}\quad  x_2 (t_0) \leq 1+30\alpha.
\end{equation}
In order to find these estimates, we combine $x_1 (t_0+1) =1$ and \eqref{eq8-4} to see that
\begin{equation}\label{eq-9-1}
\nonumber\frac{2}{3} x_1 (t_0) + \frac{1}{3} x_2 (t_0) - 10 \alpha x_1 (t_0) \leq 1.
\end{equation}
This leads to $x_1 (t_0) \leq (3-x_2 (t_0))/(2-30\alpha)$ and
\begin{equation*}
\begin{split}
\frac{1}{3} x_2 (t_0)& \leq 1 - \Big( \frac{2}{3} - 10\alpha \Big) x_1 (t_0) 
\\
&\leq 1- \Big( \frac{2}{3} -10\alpha\Big) = \frac{1}{3} + 10\alpha,
\end{split}
\end{equation*}
which implies that $x_2 (t_0) \leq 1 + 30 \alpha$.
 
\medskip

\noindent \textbf{Step 3}: We have $x_2 (t_0+1) \leq 1+30 \alpha$.

Using \eqref{eq-5-7} we deduce
\begin{equation*}
\begin{split}
\frac{1}{3} x_1 (t_0) + \Big( \frac{2}{3} + 6\alpha\Big) x_2 (t_0) 
&\leq \frac{1}{3} \Big( \frac{3-x_2 (t_0)}{2-30\alpha}\Big) + \Big(\frac{2}{3} + 6\alpha\Big) x_2 (t_0)
\\
& = \frac{1}{2-30\alpha} +\Big( \frac{2}{3} + 6\alpha - \frac{1}{6-90\alpha}\Big) x_2 (t_0). 
\end{split}
\end{equation*} 
Inserting this into \eqref{eq8-4} we find 
\begin{equation*}
\begin{split}
x_2 (t_0+1)& = \max\left[\frac{1}{3} x_1 (t_0) + \Big( \frac{2}{3} + 6\alpha\Big) x_2 (t_0), 1\right]
\\
& = \max\left[\frac{1}{2-30\alpha} +\Big( \frac{2}{3} + 6\alpha - \frac{1}{6-90\alpha}\Big) x_2 (t_0), 1\right]. 
\end{split}
\end{equation*}
Combining this with $x_2 (t_0) \leq 1+30\alpha$, we get
\begin{equation*}
\begin{split}
x_2 (t_0+1) & \leq \max\left[\frac{1}{2-30\alpha} + \Big( \frac{2}{3} + 6\alpha - \frac{1}{6-90\alpha}\Big) (1+30\alpha),1\right]
\\
& \leq 1+30\alpha.
\end{split}
\end{equation*}
Here the second inequality follows by observing the first component in the max operator equivalent form
\begin{equation*}
1 \leq (1-22\alpha +180\alpha^2) (1+30\alpha).
\end{equation*}
This can be rewritten as 
\begin{equation*}
0 \leq \alpha (1-45\alpha)(1-15\alpha),
\end{equation*}
which holds true for $ \alpha \in (0, 1/45)$.
\end{proof}

Now, we verify that the state $(x_1(t),x_2(t))$ generated by the algorithm \eqref{eq8-4} converges to an $O(\alpha)$ neighborhood of the optimal point $(1,1)$.   
\begin{thm}\label{thm8-1}
Let $g_1(x)$, $g_2(x)$ and the mixing matrix $\tilde{W}$ be defined by \eqref{eq8-2} and \eqref{eq8-3} and let $\mathbf{x}(t) = (x_1(t), x_2(t))$ be the state at $t\geq0$ generated by \eqref{eq8-4}. Then for any initial state $\mathbf{x}(0) = (x_1(0), x_2(0))$ and $\alpha \in (0,1/45)$, the state $\mathbf{x}(t)$ converges exponentially fast to the point $(1, 1/(1-18\alpha))$.
\end{thm}
\begin{proof}
By Lemma \ref{lem8-1}, we can choose $t_0$ satisfying $x_1 (t_0+1) =1$ and $x_2 (t_0 +1) \leq 1 + 30\alpha$. Note that if $\alpha<1/45$ then $1/(1-18\alpha)<1+30\alpha$. First it follows from \eqref{eq8-4} that
\begin{equation}
x_1(t+1)  =  \max \Big[\frac{2}{3}x_1(t) + \frac{1}{3}x_2(t) - 10\alpha x_1(t),~ 1\Big] =1
\end{equation}
using the following inequality 
\begin{equation*}
\frac{2}{3} x_1 (t_0 +1)+ \frac{1}{3}x_2(t_0+1)-10\alpha<1,
\end{equation*}
owing to the fact that $x_1(t_0+1) = 1$ and $x_2(t_0+1)<1+30\alpha$. Next, we use the fact $\frac{1}{3}x_1 (t_0 +1) + \frac{2}{3} x_2 (t_0 +1) \geq 1$ and the formula \eqref{eq8-4} to find the following identity
\begin{equation}\label{eq8-6}
\begin{split}
\nonumber x_2(t_0+2) &= \frac{1}{3}x_1(t_0+1) + \frac{2}{3}x_2(t_0+1)+6\alpha x_2(t_0+1)
\\
&= \frac{1}{3}x_1(t_0+1) + \frac{2+18\alpha}{3}x_2(t_0+1).
\end{split}
\end{equation}
Using this and $x_1 (t_0 +1) =1$, we get
\begin{equation}\label{eq8-7}
x_2(t_0+2) -  \frac{1}{1-18\alpha} = \frac{2+18\alpha}{3}\left[x_2(t_0+1)-\frac{1}{1-18\alpha}\right].
\end{equation}
Since $(2+18\alpha)/3<1$, it follows that $x_2(t_0+2)< x_2(t_0+1)<1+30\alpha$ for $x_2(t_0+1)>\frac{1}{1-18\alpha}$, and $1\leq x_2(t_0+2)\leq \frac{1}{1-18\alpha}$ for $1\leq x_2(t_0+1)\leq \frac{1}{1-18\alpha}$. We can conclude that $x_1(t)=1$ and $x_2(t)<1+30\alpha$ for all $t\geq t_0+1$. In addition, \eqref{eq8-7} implies that $x_2(t)$ converges to $1/(1-18\alpha)$ exponentially fast.
\end{proof}
In  the above result, we note that for $\alpha \in (0, 1/45)$ we have 
\begin{equation}
\nonumber\Big| \frac{1}{1-18\alpha}-1\Big| = \frac{18\alpha}{1-18\alpha} \leq 30 \alpha.
\end{equation}
Therefore the above result shows that  $x_2(t)$ converges to a point within $1 + O(\alpha)$, implying that the sequence $\{(x_1 (t), x_2 (t))\}_{t \geq 0}$ converges to an $O(\alpha)$-neighborhood of the optimal point $x_* =1$. This result is sharper than the convergence to an $O(\sqrt{\alpha})$-neighborhood obtained in Theorem \ref{thm2-3}.

\subsection{Half-space domain}\label{sec5-2} 
In this section, we consider the case where the domain $\Omega$ is given by the half-space $\mathbb{R}^{d-1}\times \mathbb{R}_{+}$. Specifically, we present the proof of \mbox{Theorem \ref{thm-6-10}}.  

Throughout this section, we will use  the following notations:
\begin{itemize}
\item[-] $\tilde{x}_i (t) \in \mathbb{R}^{d-1}$ denotes the first $d-1$ coordinates of $x_i (t)$.
\item[-] $y (t) \in \mathbb{R}^{d-1}$ denotes the first $d-1$ coordinates of $\bar{x}(t)$.
\item[-] $\tilde{x}_* \in \mathbb{R}^{d-1}$ denotes the first $d-1$ coordinates of $x_*$.
\item[-] $\tilde{\nabla} f(\cdot)  \in \mathbb{R}^{d-1}$ denotes the first $d-1$ coordinates of ${\nabla} f(\cdot) $.
\end{itemize}
We also suppose that  Assumptions \ref{LS}, \ref{sc}, \ref{graph}, \ref{ass-5} and \ref{ass-5-21} hold in this section. 

\begin{lem}Let $\{x_i (t)\}_{t \geq 0}$ be the sequence generated by \eqref{scheme} with constant stepsize $\alpha >0$. Then there is a positive value $J \leq \frac{2}{L+\mu}$ such that for any $\alpha \in (0,J]$ we have a  time instant $T_{\alpha} \in \mathbb{N}$ satisfying the following estimate 
\begin{equation}
\nonumber\|x_i(t) - x_*\| \leq \frac{w}{2L}
\end{equation}
and
\begin{equation}\label{eq-5-71}
\frac{1}{n}\sum_{i=1}^{n} \partial_d f_i (x_i (t)) \geq \frac{w}{2}
\end{equation}
for any $t \geq T_{\alpha}$.
\end{lem}
\begin{proof}
By Theorem \ref{thm2-3}, there exist a time $T_{\alpha} \in \mathbb{N}$ such that 
$\| x_i (t)-x_* \| =O(\sqrt{\alpha})$ for $t \geq T_{\alpha}$ if the stepsize $\alpha >0$ satisfies $\alpha \leq \frac{2}{L+\mu}$. Therefore, there exists  a suitable value $J>0$ such that $\| x_i (t)-x_* \| \leq \frac{w}{2L}$ for $\alpha \in (0,J)$.

Since a local cost function $f_i$ is $L$-smooth and $x_i (t) \in \mathbb{R}^d$ satisfies $\|x_i  (t)-x_*\| \leq \frac{w}{2L}$ for $1\leq i \leq n$ and $t \geq T_{\alpha}$, we have
\begin{equation}
\begin{split}
\nonumber\frac{1}{n}\sum_{k=1}^n \partial_d f_i (x_i  (t))& = \frac{1}{n} \sum_{i=1}^n \partial_d f_i (x_*) + \frac{1}{n}\sum_{k=1}^n \Big(\partial_d f_i (x_i  (t))-\partial f_i (x_*)\Big)
\\
&  \geq w - L \Big( \frac{w}{2L}\Big) = \frac{\omega}{2}.
\end{split}
\end{equation}
 The proof is done.
\end{proof}
To verify that a sequence $x_k(t)$ converges to an $O(\alpha)$-neighborhood of the optimal point, we first show that there exists a time $T$ such that for all $t>T$ we have $\bar{x}(t)[d] \leq O(\alpha).$ For this aim, we introduce the following lemma.
\begin{lem}\label{lem-5-41}
Take any $\alpha \in (0,J]$ and all $t\geq T_{\alpha}$. Then we have the following results.
\begin{itemize}
\item[Case 1.] Assume that $\sum_{j=1}^{n} w_{ij} x_j (t) -\eta \nabla f_i (x_i (t))$ belongs to $\Omega$ for all $1 \leq i \leq n$. Then we have 
\begin{equation}
\bar{x}(t+1)[d] \leq \bar{x}(t)[d]  -\frac{\omega\alpha}{2}.
\end{equation}
\item[Case 2.] Assume that $\sum_{j=1}^{n} w_{ij} x_j (t) -\eta \nabla f_i (x_i (t))$ does not belong to $\Omega$ for some $1 \leq i \leq n$. Then we have
\begin{equation}
\bar{x}(t+1)[d] \leq \left(\beta^{t+1} \|\x(0) - \bx(0)\|^2 + \frac{3(L^2 R_s + nR_D^*)\alpha^2}{(1-\beta)^2}\right)^{1/2}.
\end{equation} 
\end{itemize}
\end{lem}
\begin{proof} 
We begin by proving the first case, where the projection operator becomes negligible in \eqref{scheme} so that 
\begin{equation*}
x_i (t+1) =\sum_{j=1}^{n} w_{ij} x_j (t) -\alpha \nabla f_i (x_i (t)) \quad \forall~1 \leq i \leq n.
\end{equation*}
Averaging this for $1 \leq i \leq n$ gives
\begin{equation*}
\bar{x}(t+1) = \bar{x}(t) - \frac{\alpha}{n}\sum_{i=1}^{n} \nabla f_i (x_i (t)).
\end{equation*}
Considering the $d$-th coordinate of the above formula and using \eqref{eq-5-71}, it follows that
\begin{equation*}
\begin{split}
\bar{x}(t+1)[d] &= \bar{x}(t)[d] - \frac{\alpha}{n}\sum_{i=1}^{n} \partial_d f_i (x_i (t))\leq \bar{x}(t)[d]  -\frac{\omega\alpha}{2}.
\end{split}
\end{equation*} 

Next, we prove the second case. In this case, we have some $1 \leq i \leq n$ such that
\begin{equation}
\nonumber\Big(\sum_{j=1}^{n} w_{ij} x_j (t) -\eta \nabla f_i (x_i (t)) \Big)[d] \leq 0,
\end{equation}
and consequently,
\begin{equation}
\nonumber x_i (t+1)[d] = P_{\Omega}\left[\sum_{j=1}^{n} w_{ij} x_j (t) -\eta \nabla f_i (x_i (t)) \right][d] =0.
\end{equation}
Therefore
\begin{equation}
\begin{split}
\nonumber(\bar{x}(t+1)[d])^2 & = |\bar{x}(t+1)[d]- x_{i}(t+1)[d]|^2
\\
&\leq \|\bar{x}(t+1) -x_i(t+1)\|^2
\\
&\leq  \beta^{t+1} \|\x(0) - \bx(0)\|^2 + \frac{3(L^2R_s  + nR_D^*)\alpha^2}{(1-\beta)^2},
\end{split}
\end{equation} 
where we used Theorem \ref{thm2-5} for the last inequality.
\end{proof}
In the following, we show that there exists a time $T$ such that for all $t>T$, we have $\bar{x}(t)[d] \leq O(\alpha).$

\begin{prop}\label{prop-5-61}
Let $t_0 \in \mathbb{N}$ be the smallest number such that $\frac{\omega\alpha}{2}t_0  >K$ with $K>0$ defined as
\begin{equation}
\nonumber K= \Big\{ \bar{x}(T_{\alpha})[d] ,~    \Big(\beta  \|\x(0) - \bx(0)\|^2 + \frac{3(L^2R_s + nR_D^*)\alpha^2}{(1-\beta)^2}\Big)^{1/2}\Big\},
\end{equation}
 and take $t_1\in \mathbb{N}$ satisfying $\beta^{(t_1+1)}\|\x(0)-\bx(0)\|^2 \leq \frac{3(L^2R_s + nR_D^*)\alpha^2}{(1-\beta)^2}$. Then for all $t >T=t_0 + \max\{T_{\alpha}, t_1\}$, we have 
\begin{equation}
\nonumber\bar{x}(t+1)[d] \leq  \frac{\sqrt{6(L^2R_s + nR_D^*)}\alpha}{1-\beta}.
\end{equation}
\end{prop}
\begin{proof}
For each $t \geq T_{\alpha}$, the result of Lemma \ref{lem-5-41} indicates that $\bar{x}(t+1)$ is bounded by 
\begin{equation}\label{eq-5-14}
\max \Big\{ \bar{x}(t)[d]- \frac{w\alpha}{2},~ \Big(\beta^{t+1} \|\x(0) - \bx(0)\|^2 + \frac{3(L^2R_s + nR_D^*)\alpha^2}{(1-\beta)^2}\Big)^{1/2}\Big\}.
\end{equation}
From this we find that for $t\geq  T_{\alpha}$,
\begin{equation}
\nonumber\bar{x}(t)[d] \leq \max\Big\{ \bar{x}(T_{\alpha})[d],~ \Big(\beta  \|\x(0) - \bx(0)\|^2 + \frac{3(L^2R_s + nR_D^*)\alpha^2}{(1-\beta)^2}\Big)^{1/2}\Big\} = K.
\end{equation}
Therefore we have
\begin{equation}\label{eq-5-20}
\bar{x}\Big(\max\{T_{\alpha},t_1\}\Big) [d] \leq K. 
\end{equation}
Also, the estimate \eqref{eq-5-14} for $t \geq \max\{T_{\alpha}, t_1\}$ yields
\begin{equation}
\nonumber\bar{x}(t+1)[d] \leq \max \Big\{ \bar{x}(t)[d]- \frac{w\alpha}{2},~   \frac{\sqrt{6(L^2 R_s +nR_D^*)}\alpha}{1-\beta}\Big\}.
\end{equation}
Combining this with  \eqref{eq-5-20} and the fact that $\frac{w\alpha}{2}t_0 >K$, we deduce the following estimate
\begin{equation}
\bar{x}(t+1)[d] \leq \frac{\sqrt{6(L^2 R_s +nR_D^*)}\alpha}{1-\beta}\quad \textrm{for}~ t \geq t_0 +\max\{T_{\alpha},t_1\}.
\end{equation}
 It completes the proof.
\end{proof} 
 
Now we investigate the convergence property  $\tilde{x}_k(t)$ towards the point $\tilde{x}_*$ for each $1\leq k \leq n$. We first investigate the convexity property of the function $f$ in terms of its first $d-1$ coordinates in the following lemma.
\begin{lem}\label{lem-5-71}
Let $T \in \mathbb{N}$ be defined in Proposition \ref{prop-5-61} and $y(t)$ be the first $d-1$ coordinates of $\bar{x}(t)$. Then, for all $t\geq T$, we have the following inequality:
\begin{equation*}
\begin{split}
&\Big\langle y(t)-\tilde{x}_*,~\tilde{\nabla} f(\bar{x}(t))-\tilde{\nabla} f(x_*)\Big\rangle  
\\
&\geq \frac{L\mu}{2(L+\mu)}\|\bar{x}(t)-x_*\|^2 + \frac{1}{L+\mu}\|\nabla f(\bar{x}(t)) - \nabla f(x_*)\|^2\\
&\qquad\qquad\qquad\qquad\qquad\qquad\qquad- \frac{2(L+\mu)}{\mu}\cdot \frac{9L(L^2R_s + nR_D^*) \alpha^2 }{(1-\beta)^2}.
\end{split}
\end{equation*}
\end{lem}
\begin{proof}
Using the fact that $f$ is $L$-smooth and Proposition \ref{prop-5-61}, we deduce for $t >T$ the following estimate
\begin{equation*}
\begin{split}
&\Big|\Big\langle  \bar{x}(t)[d]-x_* [d],~(\partial_d f(\bar{x}(t))-\partial_d f(x_*))\Big\rangle\Big|\\
&=\Big|\Big\langle \bar{x}(t)[d],~(\partial_d f(\bar{x}(t))-\partial_d f(x_*))\Big\rangle\Big|\\
&\leq \frac{\sqrt{6L(L^2R_s + nR_D^*)}\alpha}{1-\beta}\|\bar{x}(t) - x_*\|\\
&\leq \frac{2(L+\mu)}{\mu}\cdot \frac{6L (L^2R_s + nR_D^*) \alpha^2}{4(1-\beta)^2} +\frac{L\mu}{2(L+\mu)} \|\bar{x}(t)-x_*\|^2,
\end{split}
\end{equation*}
where we used Young's inequality $ 2ab \leq \kappa a^2+ \frac{b^2}{\kappa}$ with $\kappa = \frac{2(L+\mu)}{\mu}$ for the last inequality.  Using this, we have
\begin{equation}\label{eq-5-141}
\begin{split}
&\Big\langle\bar{x}(t)-x_*,~\nabla f(\bar{x}(t))-\nabla f(x_*)\Big\rangle
\\
&=\Big\langle\tilde{x}(t)-\tilde{x}_*,~\tilde{\nabla} f(\bar{x}(t))-\tilde{\nabla} f(x_*)\Big\rangle + \Big\langle \bar{x}(t)[d]-{x}_* [d],~\partial_d f(\bar{x}(t))-\partial_d f(x_*)\Big\rangle
\\
&\leq\Big\langle\tilde{x}(t)-\tilde{x}_*,~\tilde{\nabla} f(\bar{x}(t))-\tilde{\nabla} f(x_*)\Big\rangle +\frac{ (L+\mu)}{\mu}\cdot \frac{3L (L^2R_s + nR_D^*) \alpha^2}{(1-\beta)^2}\\& \qquad\qquad\qquad\qquad\qquad\qquad\qquad\qquad\qquad +\frac{L\mu}{2(L+\mu)} \|\bar{x}(t)-x_*\|^2.
\end{split}
\end{equation}
Since the total cost function, $f$ is $L$-smooth and $\mu$-strongly convex,  we have the following  inequality (see e.g., \cite[Lemma 3.11]{B}):
$$
\Big\langle\bar{x}(t)-x_*,~\nabla f(\bar{x}(t))-\nabla f(x_*)\Big\rangle \geq \frac{L\mu}{L+\mu}\|\bar{x}(t)-x_*\|^2 + \frac{1}{L+\mu}\|\nabla f(\bar{x}(t)) - \nabla f(x_*)\|^2.
$$
Combining this with \eqref{eq-5-141}, it follows that
\begin{equation*}
\begin{split}
&\Big\langle y(t)-\tilde{x}_*,~\tilde{\nabla} f(\bar{x}(t))-\tilde{\nabla} f(x_*)\Big\rangle  
\\
&\geq \frac{L\mu}{2(L+\mu)}\|\bar{x}(t)-x_*\|^2 + \frac{1}{L+\mu}\|\nabla f(\bar{x}(t)) - \nabla f(x_*)\|^2\\
&\qquad\qquad\qquad\qquad\qquad\qquad\qquad- \frac{ (L+\mu)}{\mu}\cdot \frac{3L(L^2R_s + nR_D^*) \alpha^2}{(1-\beta)^2},
\end{split}
\end{equation*}
which proves the lemma.
\end{proof}
Now we show that the sequence $\{y(t)\}_{t \geq 0}$ converges to an $O(\alpha)$-neighborhood of the optimal point $\tilde{x}_*$.
\begin{prop}\label{prop-5-81}
Let $T\in \mathbb{N}$ be defined in Proposition \ref{prop-5-61} and assume that $\alpha \in (0, J]$.  Then there exist constants $\nu_1>0$and $\nu_2>0$ such that for $t \geq T$ we have
\begin{equation}
\nonumber\big\|y(t)-\tilde{x}_*\big\|^2 \leq (1-\nu_1\alpha)^t \big\|y(0)-\tilde{x}_*\big\|^2 +\frac{\nu_2}{\nu_1}\alpha^2. 
\end{equation}
\end{prop}
\begin{proof}
We express the first $d-1$ coordinates of formula \eqref{scheme2} as:
\begin{equation}\label{eq-5-151}
\tilde{x}_i (t+1) = \sum_{j=1}^{n} w_{ij}\tilde{x}_j (t) - \alpha \tilde{\nabla} f_i (x_i (t)).
\end{equation}
Averaging this for $1 \leq i \leq n$, we get the following formula
 \eqref{eq-5-151} as
\begin{equation}\label{eq-5-171}
y(t+1) = y(t) - \alpha\tilde{\nabla} f(\bar{x}(t)) + \alpha \left(\tilde{\nabla}f(\bar{x}(t)) - \frac{1}{n}\sum^n_{i=1}\tilde{\nabla}f_i(x_i(t))\right).
\end{equation}
Using the fact that $f_i$ is $L$-smooth and Theorem \ref{thm2-5}, we have
\begin{equation}\label{eq-5-181}
\begin{split}
\left\|\tilde{\nabla}f(\bar{x}(t)) - \frac{1}{n}\sum^n_{i=1}\tilde{\nabla}f_i(x_i(t)) \right\|^2 &\leq \frac{1}{n}\sum^n_{i=1}\left\|\tilde{\nabla}f_i(\bar{x}(t) - \tilde{\nabla} f_i(x_i(t))\right\|^2 \\
&\leq \frac{L^2}{n}\sum^n_{i=1}\left\|\bar{x}(t) - x_i(t)\right\|^2\\
&\leq L^2 \beta^{t }\|\x(0)-\bx(0)\|^2 + \frac{3L^2 (L^2R_s + nR_D^*)\alpha^2}{1-\beta} \\
&\leq \frac{6L^2 (L^2R_s + nR_D^*)\alpha^2}{1-\beta},
\end{split}
\end{equation}
where we used the fact that $T >t_1$ with $t_1$ defined in Proposition \ref{prop-5-61} for the last inequality. Using \eqref{eq-5-171}, Young's inequality and \eqref{eq-5-181}, it follows that
\begin{equation}\label{eq-5-191}
\begin{split}
&\|y(t+1) - \tilde{x}_*\|^2\\ =& \left\|y(t) - \alpha f(\bar{x}(t)) + \alpha \left(\tilde{\nabla}f(\bar{x}(t)) - \frac{1}{n}\sum^n_{i=1}\tilde{\nabla}f_i(x_i(t))\right)\right\|^2 \\
\leq &(1+c\alpha)\left\| y(t) - \alpha \tilde f(\bar{x}(t)) - \tilde{x}_*\right\|^2 + \left(1+\frac{1}{c\alpha} \right)\frac{6L^2(L^2R_s + nR_D^*)}{(1-\beta)^2}\alpha^4,
\end{split}
\end{equation}
where we have let $c = \frac{L\mu}{2(L+\mu)}$.
Next, we estimate the first term in the last inequality in \eqref{eq-5-191}. By the first optimality condition, we have $\tilde{\nabla}f(x_*) = 0$. Using this it follows that
\begin{equation*}
\begin{split}
&\|y(t) -\alpha \tilde{\nabla} f(\bar{x}(t))-\tilde{x}_* \|^2 \\
&=\Big\|y(t)-\tilde{x}_*  -\alpha \Big(\tilde{\nabla} f(\bar{x}(t))-\tilde{\nabla} f(x_*)\Big)\Big\|^2 
\\
&=\big\|y(t)-\tilde{x}_*\big\|^2  + \alpha^2\Big\|\tilde{\nabla} f(\bar{x}(t))-\tilde{\nabla} f(x_*)\Big\|^2  -2 \alpha \Big\langle y(t)-\tilde{x}_*,~ \tilde{\nabla} f(\bar{x}(t))-\tilde{\nabla} f(x_*)\Big\rangle.
\end{split}
\end{equation*}
Applying Lemma \ref{lem-5-71} to the right -hand side, we get
\begin{equation*}
\begin{split}
&\| y(t) -\alpha\tilde{\nabla} f(\bar{x}(t))-\tilde{x}_* \|^2 
\\
 &\geq \Big(1 - \frac{L\mu}{L+\mu} \alpha\Big)\big\|y(t)-\tilde{x}_*\big\|^2  + \Big( \alpha^2 - \frac{2\alpha}{L+\mu}\Big)\Big\| \Big(\tilde{\nabla} f(\bar{x}(t))-\tilde{\nabla} f(x_*)\Big)\Big\|^2  
 \\
 &\quad  + \frac{2(L+\mu)}{\mu}\cdot \frac{3L(L^2R_s + nR_D^*) \alpha^3}{(1-\beta)^2}.
\end{split}
\end{equation*}
This inequality with the condition $\alpha \leq \frac{2}{L+\mu}$ gives
\begin{equation}\label{eq-5-201}
\begin{split}
&\|y(t) -\eta \tilde{\nabla} f(\bar{x}(t))-\tilde{x}_* \|^2 
\\ &= \Big(1 - \frac{L\mu}{L+\mu} \alpha\Big)\big\|y(t)-\tilde{x}_*\big\|^2 + \frac{2(L+\mu)}{\mu}\cdot \frac{3L (L^2R_s + nR_D^*) \alpha^3}{(1-\beta)^2}.
\end{split}
\end{equation}
Putting \eqref{eq-5-201} into \eqref{eq-5-191} we obtain
\begin{equation*}
\begin{split}
&\|y(t+1)- \tilde{x}_*\|^2
\\
&\leq (1+c\alpha)\Big(1 - \frac{L\mu}{L+\mu} \alpha\Big)\big\|y(t)-\tilde{x}_*\big\|^2 + (1+c\alpha)\frac{2(L+\mu)}{\mu}\cdot \frac{3L (L^2R_s + nR_D^*) \alpha^3}{(1-\beta)^2} \\
&\qquad\qquad\qquad\qquad\qquad +\left(1+\frac{1}{c\alpha} \right)\frac{6L^2(L^2R_s + nR_D^*) }{(1-\beta)^2}\alpha^4
\\
&\leq \Big(1 - \frac{L\mu}{2(L+\mu)} \alpha\Big)\big\|y(t)-\tilde{x}_*\big\|^2 + (1+c )\frac{2(L+\mu)}{\mu}\cdot \frac{3L (L^2R_s + nR_D^*) \alpha^3}{(1-\beta)^2} \\
&\qquad\qquad\qquad\qquad\qquad +\left(1+\frac{1}{c} \right)\frac{6L^2(L^2R_s + nR_D^*) }{(1-\beta)^2}\alpha^3,
\end{split}
\end{equation*}
where the second inequality follows using that $\alpha \leq 1$.
This estimate can be written in the following form
\begin{equation}
\nonumber\|y(t+1)- \tilde{x}_*\|^2
\leq (1-\nu_1\alpha )\big\|y(t)-\tilde{x}_*\big\|^2 + \nu_2\alpha^3,
\end{equation}
where $\nu_1 := \frac{L\mu}{2(L+\mu)}$ and
\begin{equation}
\nonumber\nu_2:= (1+c)\frac{2(L+\mu)}{\mu}\cdot \frac{3L (L^2R_s + nR_D^*)  }{(1-\beta)^2}+\left(1+\frac{1}{c} \right)\frac{6L^2(L^2R_s + nR_D^*) }{(1-\beta)^2}.
\end{equation} 
Using this iteratively, we get
\begin{equation}
\begin{split}
\nonumber\big\|y(t)-\tilde{x}_*\big\|^2 &\leq (1-\nu_1\alpha)^t \big\|y(0)-\tilde{x}_*\big\|^2 +\nu_2\alpha^3 \sum_{s=0}^{t-1} (1-\nu_1\alpha)^s
\\
&\leq (1-\nu_1\alpha)^t \big\|\tilde{x}(0)-\tilde{x}_*\big\|^2 +  \frac{\nu_2}{\nu_1}\alpha^2,
\end{split}
\end{equation}
which completes the proof.
\end{proof}
We are ready to prove Theorem \ref{thm-6-10}.
\begin{proof}[Proof of Theorem \ref{thm-6-10}]
First we consider the case that $\alpha \in (0,J)$ since $x_* \in \partial (\mathbb{R}^{d-1}\times \mathbb{R}^+)$, we have the following equality
\begin{equation}\label{eq-5-33}
\|x_i(t) - x_*\|^2 =  \|\tilde{x}_i(t) - \tilde{x}_*\|^2 + \|x_i(t)[d]\|^2.
\end{equation}
By Theorem \ref{thm2-5} and Proposition \ref{prop-5-81}, for any $t\geq T$, we have
\begin{equation*}
\begin{split} 
&  \|\tilde{x}_i(t) - \tilde{x}_*\|^2  
\\
&\leq 2\|\tilde{x}_i(t) -y(t)\|^2 +2\|y(t) - \tilde{x}_*\|^2
\\
&\leq 2\beta^t\|\x(0) -\bx(0)\|^2 + 2 (1-\nu_1\alpha)^t\|y(0)-\tilde{x}_*\|^2 +   \left( \frac{ 6(L^2R_s + nR_D^*) }{(1-\beta)^2} + \frac{2 \nu_2}{\nu_1} \right)\alpha^2,
\end{split}
\end{equation*}
By Proposition  \ref{prop-5-61}, 
\begin{equation*}
\begin{split}
 \|x_i(t)[d]\|^2 &\leq2\|x_i(t)[d] -\bar{x}(t)\|^2 +2\|\bar{x}(t)\|^2  \\
&\leq 2\beta^t\|\x(0) -\bx(0)\|^2  +  \frac{12(L^2R_s + nR_D^*)  }{(1-\beta)^2}   \alpha^2.
\end{split}
\end{equation*}
Inserting the above two estimates in \eqref{eq-5-33} we get
\begin{equation}
\begin{split}
\nonumber&\|x_i(t) - x_*\|^2 \leq  4\beta^t\|\x(0) -\bx(0)\|^2 + 2 (1-\nu_1\alpha)^t\|y(0)-\tilde{x}_*\|^2 + \\
&\qquad\qquad\qquad+\left( \frac{ 18(L^2R_s + nR_D^*) }{(1-\beta)^2} + \frac{2 \nu_2}{\nu_1} \right)\alpha^2.
\end{split}
\end{equation}
This gives for $\alpha \in (0,J)$ the following estimate
\begin{equation}
\nonumber\underset{t \rightarrow \infty}{\textrm{limsup}} \,\|x_i(t) - x_*\|^2  = O(\alpha^2).
\end{equation}
For $J \leq \alpha \leq \frac{2}{L+\mu}$, we may use the result of Theorem \ref{thm2-3} to derive 
\begin{equation}
\nonumber\underset{t \rightarrow \infty}{\textrm{limsup}}  \, \|x_i(t) - x_*\|^2  \leq C\alpha \leq \frac{C}{J} \alpha^2.
\end{equation}
Combining the above two estimates completes the proof.
\end{proof}

\section{Simulations}\label{sec6}
In this section, we conduct numerical experiments to validate the DPG algorithm. These experiments include non-negative least squares (Subsection \ref{sec6-1}), constrained logistic regression (Subsection \ref{sec6-3}), and an example in \ref{sec5-1} (Subsection \ref{sec6-2}). In Subsections \ref{sec6-1} and \ref{sec6-3}, we construct a connected graph based on the Watts and Strogatz model \cite{WS}, and employ a Laplacian-based constant edge weight matrix denoted as $W$ \cite{S, XB}. This matrix is defined by $W = \mathbf{I} - \mathcal{L}/\tau$ where $\mathcal{L}$ is the Laplacian matrix and $\tau$ is a constant that satisfies $\tau>\frac{1}{2}\lambda_{\max}(\mathcal{L})$. If $\lambda_{\max}(\mathcal{L})$ is not available, one can use $\tau=\max_{1\leq i \leq n}D_{ii}(\mathcal{L})$ instead. This construction satisfies Assumptions \ref{graph} and \ref{ass-1-1}. In addition, we compare the DPG algorithm with the algorithm presented in \cite{DMDT}, a version of the DIGing algorithm \cite{nedic2017} for the constrained problem. We refer to the algorithm in \cite{DMDT} as P-DIGing in Subsections \ref{sec6-1} and \ref{sec6-3}.

\subsection{Non-negative least squares}\label{sec6-1} We consider the following decentralized non-negative least squares problem with $n$ agents
$$
\min_{x\in\Omega} \frac{1}{n} \sum^n_{i=1}\frac{1}{2}\|q_i - p_i^Tx\|^2,
$$ 
where $\Omega = \{\left(x[1],x[2],\cdots,x[d]\right)~|~x[k]\in \mathbb{R}_+\}.$ In this case, the projection operator $\mathcal{P}_{\Omega}$ for a point $x=\left(x[1],x[2],\cdots, x[d]\right)$ is defined by
\begin{equation*}
\mathcal{P}_{\Omega}[x] = 
\left\{\begin{array}{ll}
0 \ \text{if $x[k] < 0$},&\\
x[k], \ \text{otherwise},& 
\end{array}\right.
\end{equation*}
for all $1\leq k \leq d$. We initialized the points $x_i(0)$ by choosing independent random variables generated from a standard Gaussian distribution after applying the projection operator. The variables $p_i\in \mathbb{R}^{d\times p}$ and $q_i\in\mathbb{R}^{p}$ are randomly chosen from the uniform distribution on $[0,1]$. In this simulation, we set the problem dimensions and the number of agents as $d=10$, $p=5$, and $n=30$. In addition, each agent has four neighbors based on the Watts and Strogatz model. We consider the relative convergence error 
$$
R(t) = \frac{\sum^n_{i=1}\|x_i(t) - x_*\|}{\sum^n_{i=1}\|x_i(0) - x_*\|},
$$
and examine the graphs of $R(t)$ with various constant stepsizes and diminishing stepsizes. We compute the values for the constants $L$ and $\mu$ by
$$
L = \max_{1\leq i \leq n} \lambda_1 \Big(p_ip_i^T\Big) ~ \text{and} ~ \mu =\lambda_n \Big( \frac{1}{n} \sum_{i=1}^n p_i p_i^T\Big),$$
where $\lambda_k$ denotes the $k$-th largest eigenvalue  of a matrix. 

For constant stepsizes, we consider $\alpha\in \{0.5/(L+\mu), 1/(L+\mu), 2/(L+\mu), 3/(L+\mu), 1/\tilde{L}\}$ where the stepsize $1/\tilde{L}$ corresponds to the theoretical critical stepsize mentioned in \cite{LLSX} with $\tilde{L}=\frac{1}{n}\sum^n_{i=1}\lambda_1 \Big(p_ip_i^T\Big)$. 
As for the decreasing stepsize  $\alpha (t)  = v/(t+w)^p$, we choose the constant $w$ can be determined by $w = (L+\mu)^{1/p}$ obeying the condition of Theorem \ref{thm2-6} and test for $p \in\{0.25, 0.5, 0.75, 1\}$.
The numerical results for constant and diminishing stepsizes are illustrated in Figure \ref{figure1-1}. 
\begin{figure}[!htbp]
\begin{center}
\includegraphics[width=6.45cm]{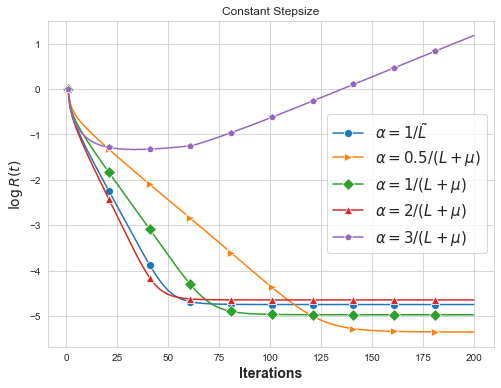}
\includegraphics[width=6.45cm]{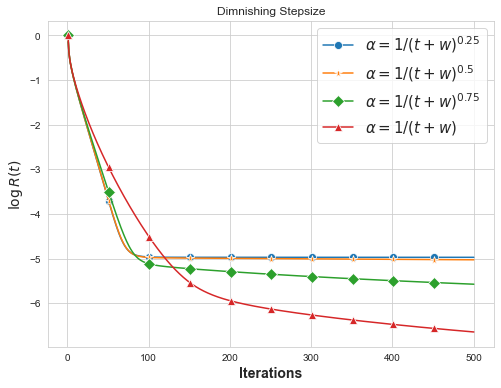}
\end{center}
\caption{The graphs of $\log R(t)$ under various choices of constant stepsizes (left), diminishing stepsizes (right). } 
\label{figure1-1}
\end{figure}
We observe that with a constant stepsize, the value of $R(t)$ converges to a neighborhood of the solution, and the size of the neighborhood depends on the chosen stepsize as long as it meets the condition outlined in Theorem \ref{thm2-3}. It is also shown that the DPG with the stepsize $\alpha = 3/(L+\mu)$ diverges. Additionally, the right figure in Figure \ref{figure1-1} shows that with diminishing stepsizes, the values of $R(t)$ converge to zero, which is expected in Theorem \ref{thm2-6}.

Subsequently, we conduct a comparative analysis between the constant stepsize, the diminishing stepsize, and the P-DIGing algorithm \cite{DMDT}. It was shown that the P-DIGing algorithm with a constant stepsize converges to the optimizer. This is a stronger point in comparison with the DPG algorithm, but the range of the stepsize for the convergence for the P-DIGing tends to be smaller than that of the DPG, as shown in the graphs of Figure \ref{figure1-2}. This motivates us to design a hybrid algorithm 'the DPG+P-DIGing', which employs the DPG with a constant stepsize during the initial stages, and then changes to perform the P-DIGing algorithm. We empirically determine an optimal constant step size such that the P-DIGing algorithm converges fast. 
The numerical results for constant and diminishing stepsizes are illustrated in Figure \ref{figure1-2}.
\begin{figure}[!htbp]
\begin{center}
\includegraphics[width=6.45cm]{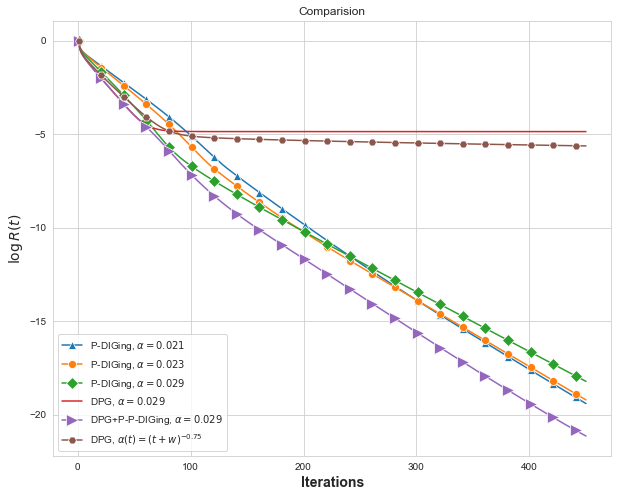}
\end{center}
\caption{The grphs of $\log R(t)$ with P-DIGing, DPG using a constant step size, DPG+P-DIGing, and DPG with a diminishing step size.} 
\label{figure1-2}
\end{figure}
We observe that the DPG exhibits faster convergence during iterations $t$ up to around $75$ while the P-DIGing converges exactly to the optimizer. Combining these two strength points, the DPG+P-DIGing algorithm is shown to have the best performance.

\subsection{Constrained logistic regression}\label{sec6-3}
We consider the following decentralized logistic regression problem with the MNIST dataset \cite{deng2012mnist}:
\begin{equation*}
\min_{x\in\Omega}\sum^n_{i=1}f_i(x),
\end{equation*}
where $f_i(x) = \sum^k_{j=1}\log[1+\exp\left((-x^T\tau_j)\phi_j\right)] + \frac{\alpha}{2}\|x\|^2$ and $\Omega = [-1,5]^{784}$.
Here $k$ represents the number of data points, $\tau_j\in\mathbb{R}^{784}$ is the feature vector and $\phi_j \in\{-1,1\}$ is the corresponding class. For the experiment, we select only the number 1 and 2 classes in the MNIST data and rename these classes as -1 and 1, respectively. In this experiment, we set the number of agents as $n=20$, with each agent having four neighbors. Additionally, we randomly select 1000 data points from each of class 1 and class 2, respectively, and equally divide this data among the agents. Consequently, each agent is assigned 50 data points from both class 1 and class 2. To obtain $x_*$, we employ the centralized projected gradient descent algorithm. We measure the quantity $R(t)$ for DPG with a constant stepsize as well as the P-DIGing algorithm (see Figure \ref{figure2-1}).
\begin{figure}[!htbp]
\begin{center}
\includegraphics[width=7cm]{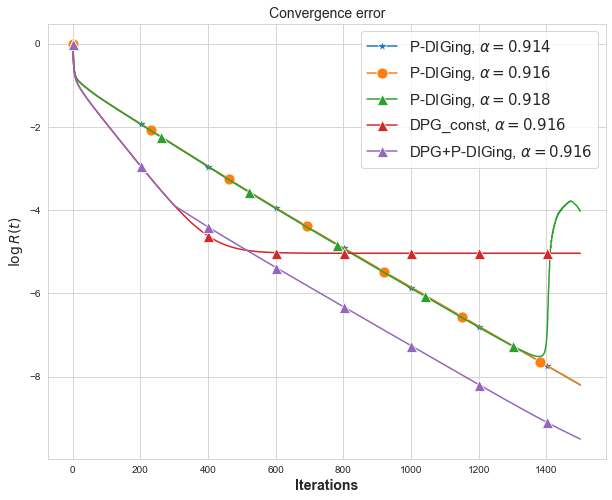}
\end{center}
\caption{The graph of $\log R(t)$ with P-DIGing, DPG using a constant step size, DPG+P-DIGing.}
\label{figure2-1}
\end{figure}
Similar to Subsection \ref{sec6-1}, we empirically select an appropriate constant step size for P-DIGing. As we expected, the DPG algorithm exhibits a faster convergence rate than the P-DIGing algorithms for first iterations up to around $500$. Consequently, the DPG+P-DIGing algorithm is shown to have a faster convergence compared to the P-DIGing algorithm.

\subsection{The example in Subsection \ref{sec5-1}}\label{sec6-2}
We provide a numerical test for the example considered in Subsection \ref{sec5-1}. Specifically, we test the convergence property of the algorithm \eqref{eq8-4}. To verify the result of Theorem \ref{thm8-1}, we consider the sequence $(x_1 (t), x_2(t))$ of \eqref{eq8-4} and the following measure
\begin{equation}
\nonumber\tilde{R}(t) = \frac{(x_1 (t)-1)^2 + \Big(x_2 (t) - \frac{1}{1-18\alpha}\Big)^2}{(x_1 (0)-1)^2 + \Big(x_2 (0) - \frac{1}{1-18\alpha}\Big)^2}\quad\textrm{for}~ t \geq 0.
\end{equation}
\begin{figure}[!htbp]
\includegraphics[width=6.45cm]{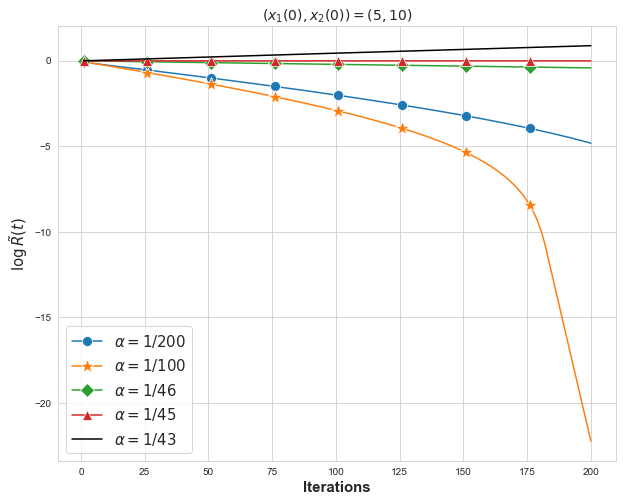}
\includegraphics[width=6.45cm]{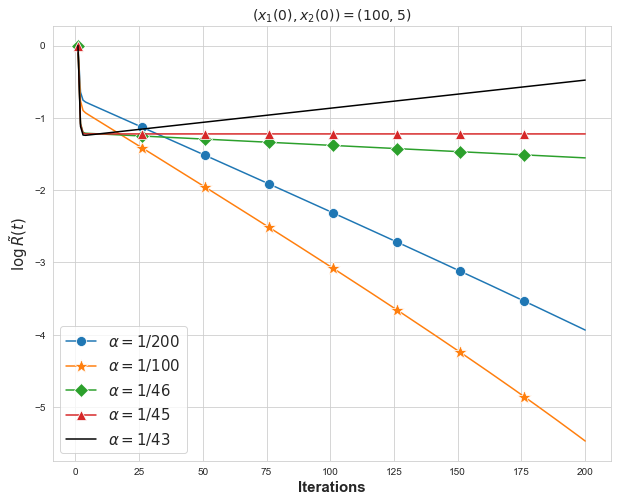}
\caption{The left and right graphs represent the value $\tilde{R}(t)$ for the initial values $(x_1 (0), x_2(0))=(5,10)$ and $(100,5)$, respectively.}
\label{figure5}
\end{figure}
We test with the stepsizes $\{1/200, 1/100, 1/46, 1/45, 1/43\}$ and the initial values given as
 \begin{equation}
 \nonumber(x_1 (0), x_2 (0)) \in \{(5,10), (100,5)\}.
\end{equation}
The graph of the measure $\tilde{R}(t)$ is provided in Figure  \ref{figure5}. The result shows that the algorithm \eqref{eq8-4} converges to the value $(1, 1/(1-18\alpha))$ for the stepsizes $\{1/200, 1/100, 1/46\}$ as expected by Theorem \ref{thm8-1}. Meanwhile, the algorithm \eqref{eq8-4} does not converge for the stepsize $\alpha \in\{1/43,1/45\}$ which is not supported in the interval $(0,1/45)$ guaranteed by Theorem \ref{thm8-1}. These results verify the sharpness of the result of Theorem \ref{thm8-1}.

\section{Conclusion}In this paper, we have introduced new convergence estimates for the decentralized projected gradient method. Our findings guarantee that when the algorithm employs a fixed step size $\alpha$ (where $\alpha$ is a positive constant) below a certain threshold, it converges to an  $O(\sqrt{\alpha})$ neighborhood of the optimal solution exponentially fast. Furthermore, we have established exact convergence results for the case with diminishing stepsize  $\alpha(t) = \frac{v}{(t + w)^p}$ for $p\in(0, 1]$. We also further improved the convergence result up to the $O(\alpha)$-error for a one-dimensional example and for a general class of functions defined on  the domain $\Omega = \mathbb{R}^{d-1} \times \mathbb{R}_{+}$ with any $d \geq 1$
\appendix
\section{Proof of Sequential estimates}\label{secA}
In this appendix, we provide the proofs for the sequential estimates of Proposition \ref{prop-3-1} and Proposition \ref{prop-3-2}. Proposition \ref{prop-3-1} is concerned about the consensus estimate and is established based on a well-known result of Lemma \ref{lem-1-1} regarding the consensus matrix $W$ along with Lemma \ref{lem-1-7} related to projection operator and Lemma \ref{lem-1-5} exploiting the $L$-smooth property of the cost functions.  Proposition \ref{prop-3-2} provides the sequential estimate between the points $x_i (t)$ and the optimal point $x_*$. It is established based on the estimate \eqref{eq-3-16} and Lemma \ref{lem-1-5} along with Lemma \ref{lem-1-4} exploiting the strong convexity of the aggregate cost function $f$. The overall flow of the proofs is summarized in \mbox{Figure \ref{diag}.}


\begin{figure}[!htbp]
\begin{center}
\includegraphics[width=9cm]{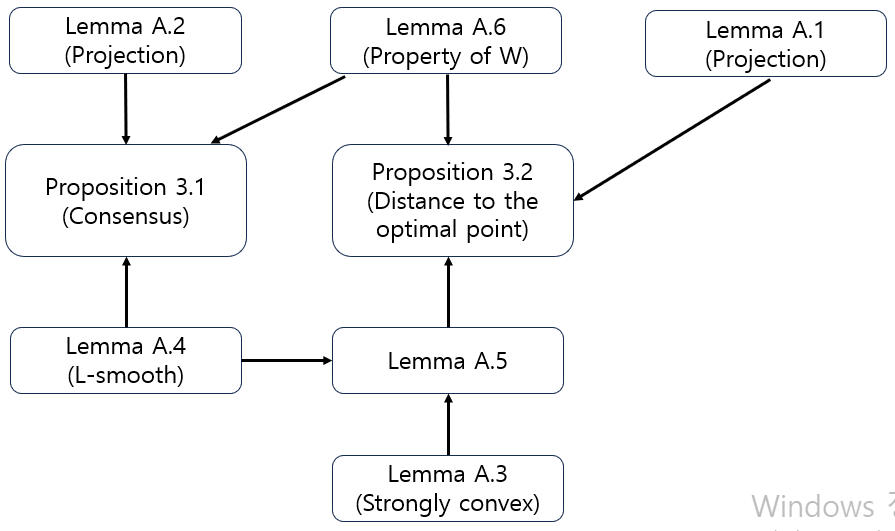}
\end{center}
\caption{The overall flows of the proofs for Proposition \ref{prop-3-1} and Proposition \ref{prop-3-2}.}
\label{diag}
\end{figure}

\subsection{Projection operator}\label{secA-1}
To obtain sequential estimates, it is important to know the properties of {the projection operator $\mathcal{P}_{\Omega}$, which is defined in \eqref{projection}.}
This projection operator has the following properties:
\begin{lem}\label{lem-1-2}
Let $\Omega\subset \mathbb{R}^d$ be convex and closed. 
\begin{enumerate}
\item For any $x, y \in \mathbb{R}^d$, we have
\begin{equation}\label{eq-2-2}
\|\mathcal{P}_{\Omega} [x]-\mathcal{P}_{\Omega}[y]\| \leq \|x-y\|.
\end{equation}
\item For any $x\in \mathbb{R}^d$ and $y\in\Omega$, we have
\begin{equation}\label{eq-2-1}
\|\mathcal{P}_{\Omega}[x] - y\| \leq \|x-y\|.
\end{equation}
\end{enumerate}
\end{lem}
\begin{proof}
We refer to Theorem 1.5.5 in \cite{FP} for the proof of estimate \eqref{eq-2-2}. Taking $y\in\Omega$ in \eqref{eq-2-2} leads to the estimate \eqref{eq-2-1}. 
\end{proof}

\begin{lem}\label{lem-1-7}
Let $\Omega\subset \mathbb{R}^d$ be convex and closed. Then, for any $x_1,\cdots, x_n\in \mathbb{R}^d$, we have
\begin{equation*}
\sum^n_{i=1} \Bigg\| \mathcal{P}_{\Omega}[x_i] - \frac{1}{n}\sum^n_{j=1}\mathcal{P}_{\Omega}[x_j]\Bigg\|^2\leq \sum^n_{i=1} \left\| x_i - \frac{1}{n}\sum^n_{j=1}x_j \right\|^2.
\end{equation*}
\end{lem}
\begin{proof}
For given $\{a_i\}_{i=1}^n \subset \mathbb{R}^d$, consider the function $F: \mathbb{R}^d \rightarrow \mathbb{R}$ defined by
\begin{equation*}
F(x) = \sum_{i=1}^n \|a_i -x\|^2,
\end{equation*}
This function is minimized when $x = \frac{1}{n} \sum_{i=1}^n a_i$, and using this we find
\begin{equation*}
\sum^n_{i=1} \Bigg\| \mathcal{P}_{\Omega}[x_i] - \frac{1}{n}\sum^n_{j=1}\mathcal{P}_{\Omega}[x_j]\Bigg\|^2\leq \sum^n_{i=1} \Bigg\| \mathcal{P}_{\Omega}[x_i] - \mathcal{P}_{\Omega}\bigg[\frac{1}{n}\sum^n_{j=1}x_j\bigg]\Bigg\|^2.
\end{equation*}
Combining this with \eqref{eq-2-2}, we get the desired inequality.
\end{proof}
\subsection{Preparation}\label{secA-2}
In this subsection, we present several results to prove Propositions \ref{prop-3-1} and \ref{prop-3-2}. First, we introduce Lemma \ref{lem-1-4}, which will serve as a foundation for deriving Lemma \ref{lem-3-8}. Subsequently, we introduce Lemmas \ref{lem-1-5}, \ref{lem-3-8}, and \ref{lem-1-1}, all of which will be used to establish Propositions \ref{prop-3-1} and \ref{prop-3-2} in Subsection \ref{secA-3}.
\begin{lem}\label{lem-1-4}
Suppose that Assumptions \ref{LS} and \ref{sc} hold. If $\alpha(t)\leq \frac{2}{L+\mu}$ for all $t\geq 0 $, then we have
\begin{equation}\label{eq3-6}
\left\| \bar{x}(t)-x_*- \frac{\alpha(t)}{n}\sum^n_{i=1}( \nabla f_i(\bar{x}(t))- \nabla f_i(x_*))\right\|^2    \leq \bigg(1-\frac{\mu\alpha(t)}{2} \bigg)^2\| \bar{x}(t)-x_*\|^2 .
\end{equation}
\end{lem}

\begin{proof}
We expand the left hand side in \eqref{eq3-6} as
\begin{equation}\label{eq3-5}
\| \bar{x}(t)-x_*\|^2 - 2\alpha(t)\Big\langle \bar{x}(t)-x_*, \nabla f(\bar{x}(t))- \nabla f(x_*)\Big\rangle + \alpha(t)^2\left\| \nabla f(\bar{x}(t))- \nabla f(x_*)\right\|^2,
\end{equation}
where $f(x) = \frac{1}{n}\sum^n_{i=1} f_i(x).$ Note that since $f$ is $L$-smooth and $\mu$-strongly convex Assumptions \ref{LS} and \ref{sc}, we have the following  inequality (see e.g., \cite[Lemma 3.11]{B}):
$$
\Big\langle\bar{x}(t)-x_*,~\nabla f(\bar{x}(t))-\nabla f(x_*)\Big\rangle \geq \frac{L\mu}{L+\mu}\|\bar{x}(t)-x_*\|^2 + \frac{1}{L+\mu}\|\nabla f(\bar{x}(t)) - \nabla f(x_*)\|^2.
$$
Putting the above inequality in \eqref{eq3-5}, we get
\begin{equation*}
\begin{split}
&\left \| \bar{x}(t)-x_*- \frac{\alpha(t)}{n}\sum^n_{i=1} ( \nabla f_i(\bar{x}(t))- \nabla f_i(x_*))\right\|^2 \\
&\leq \bigg(1-\frac{2L\mu\alpha(t)}{L+\mu} \bigg)\| \bar{x}(t)-x_*\|^2  + \alpha(t)\bigg(\alpha(t)-\frac{2}{L+\mu}\bigg)\|\nabla f (\bar{x}(t))-\nabla f(x_*)\|^2.
\end{split}
\end{equation*}
Using the assumption $\alpha(t)\leq \frac{2}{L+\mu}$ and $2L\geq L+\mu$, we then have
\begin{equation*}
\begin{split}
\left\| \bar{x}(t)-x_*- \frac{\alpha(t)}{n}\sum^n_{i=1}( \nabla f_i(\bar{x}(t))- \nabla f_i(x_*))\right\|^2 & \leq \bigg(1-\frac{2L\mu\alpha(t)}{L+\mu} \bigg)\| \bar{x}(t)-x_*\|^2 
\\
&\leq \bigg(1-\frac{\mu\alpha(t)}{2} \bigg)^2\| \bar{x}(t)-x_*\|^2,
\end{split}
\end{equation*}
which completes the proof.
\end{proof}

\begin{lem}\label{lem-1-5}
Suppose that Assumption \ref{LS} holds. For $(x_1, \cdots, x_n ) \in \mathbb{R}^{dn}$ and $\bar{x} = \frac{1}{n} \sum_{k=1}^n x_k$  we have
\begin{equation}\label{eq3-7}
\sum^n_{i=1}\|  \nabla f_i({x}_i) - \nabla f_i(\bar{x})  \|^2 \leq L^2\|\mathbf{x} - \bar{\mathbf{x}}\|^2\end{equation}
and
\begin{equation}\label{eq3-8}
\sum_{i=1}^n \Big\| \nabla f_i({x}_i) - \frac{1}{n} \sum_{l=1}^n \nabla f_l ({x}_l) \Big\|^2 \leq  3L^2 \| \mathbf{x} - \bar{\mathbf{x}}\|^2 +3L^2 \| \bar{\mathbf{x}}- \mathbf{x}_*\|^2 + 3nD^2.
\end{equation}
\end{lem}
\begin{proof}
By $L$-smoothness of a local cost function $f_i$, we have
\begin{equation*}
\sum^n_{i=1}\|  \nabla f_i({x}_i) - \nabla f_i(\bar{x})  \|^2 \leq L^2 \sum^n_{i=1}\|  {x}_i-\bar{x} \|^2,
\end{equation*}
which directly implies \eqref{eq3-7}.
Next, we prove \eqref{eq3-8}.  Note that for any $a=(a_1,\cdots,a_n)$ $\in\mathbb{R}^n$, we have 
\begin{equation*}
\begin{split}
\sum^n_{i=1}\left\|a_i -\frac{1}{n}\sum^n_{l=1}a_l\right\|^2&=\sum^n_{i=1}\left(\|a_i\|^2 -2\left\langle a_i, \frac{1}{n}\sum^n_{l=1}a_l\right\rangle +\frac{1}{n^2}\left\|\sum^n_{l=1}a_l\right\|^2\right)\\
&=\sum^n_{i=1} \|a_i\|^2 +\left(\frac{1}{n^2}-\frac{2}{n}\right)\left\|\sum^n_{l=1}a_l\right\|^2 \\
&\leq \sum^n_{i=1} \|a_i\|^2.
\end{split}
\end{equation*}
Using this, it follows that
\begin{equation}\label{eq-3-11}
\begin{split}
\sum_{i=1}^n \left\| \nabla f_i({x}_i)  - \frac{1}{n} \sum_{l=1}^n \nabla f_l ({x}_l)\right\|^2 &\leq \sum^n_{i=1}  \| \nabla f_i({x}_i(t))\|^2.
\end{split}
\end{equation}
By the triangle inequality, it follows that
\begin{equation*}
\begin{split}
\|\nabla f_i({x}_i)\|^2 &\leq 3\|\nabla f_i({x}_i) - \nabla f_i(\bar{x}) \|^2 + 3\|  \nabla f_i(\bar{x}) - \nabla f_i(x_*)\|^2 + 3\|\nabla f_i(x_*) \|^2.
\end{split}
\end{equation*}
This, together with $L$-smoothness of $f_i$, gives
\begin{equation*}
\sum^n_{i=1}\|\nabla f_i({x}_i)\|^2\leq 3L^2 \| \mathbf{x} - \bar{\mathbf{x}}\|^2 +3L^2 \| \bar{\mathbf{x}}- \mathbf{x}_*\|^2 + 3nD^2.
\end{equation*}
Combining this with \eqref{eq-3-11}, we have the desired estimate as follows:
\begin{equation*}
\begin{split}
\sum_{i=1}^n \left\| \nabla f_i({x}_i)  - \frac{1}{n} \sum_{l=1}^n \nabla f_l ({x}_l)\right\|^2 &\leq \sum^n_{i=1}  \| \nabla f_i({x}_i(t))\|^2.\\
&\leq 3L^2 \| \mathbf{x} - \bar{\mathbf{x}}\|^2 +3L^2 \| \bar{\mathbf{x}}- \mathbf{x}_*\|^2 + 3nD^2.
\end{split}
\end{equation*}
\end{proof}
  

\begin{lem}\label{lem-3-8}
Suppose that Assumptions \ref{LS} and \ref{sc} hold.  If a sequence $\{\alpha(t)\}_{t\geq0}$  satisfy $\alpha(t)\leq \frac{2}{L+\mu}$, then the sequence $\{x_i(t)\}_{t\geq0}$ generating by \eqref{scheme} for all $1\leq i \leq n$ satisfies the following inequality 
\begin{equation}\label{eq-3-50}
\begin{split}
& n\left\|\bar{x}(t)-x_* -\alpha(t)\left(\frac{1}{n}\sum^n_{i=1} \nabla f_i({x}_i(t)) -  \nabla f(x_*)\right) \right\|^2\\
&\quad \leq \bigg(1- \frac{\mu \alpha (t)}{2}\bigg)\|\bar{\mathbf{x}}(t) - \mathbf{x}_* \|^2 + \bigg(L^2\alpha(t)^2+\frac{4L^2\alpha(t)}{\mu}\bigg) \|\mathbf{x}(t)-\bar{\mathbf{x}}(t) \|^2.
\end{split}
\end{equation}
\end{lem}
\begin{proof} 
By Young's inequality, for any $u$ and $v$ in $\mathbb{R}^n$ we have
\begin{equation*}\label{eq3-11}
\|u+v\|^2 \leq (1+\eta) \|u\|^2 + \left( 1+ \frac{1}{\eta}\right) \|v\|^2
\end{equation*}
for all $\eta >0$. Using Young's inequality with $\eta = \mu\alpha(t)/4$, we obtain the following inequality:
\begin{equation*}
\begin{split}
& n\left\|\bar{x}(t)-x_* -\alpha(t) \left(\frac{1}{n}\sum^n_{i=1}\nabla f_i(x_i(t)) -  \nabla f(x_*)\right) \right\|^2 \\
&= n\left\|\bar{x}(t) - x_* - \alpha(t) \left(\nabla f(\bar{x}(t))-\nabla f(x_*)\right) + \alpha(t) \left( \nabla f(\bar{x}(t)) - \frac{1}{n}\sum^n_{i=1}\nabla f_i(x_i(t)) \right) \right\|^2\\
&\leq  n\left(1+\frac{\mu\alpha(t)}{4}\right)\|X(t)\|^2+ n\left(1+\frac{4}{\mu\alpha(t)}\right)\alpha(t)^2 \|Y(t)\|^2,
\end{split}
\end{equation*}
where 
\begin{equation*}
\begin{split}
X(t) &=\bar{x}(t)-x_* -\alpha(t)(\nabla f(\bar{x}(t))-\nabla f(x_*)\\
Y(t) &= \alpha(t)\left(\nabla f(\bar{x}(t))-\frac{1}{n}\sum^n_{i=1}\nabla f_i(x_i(t))\right).
\end{split}
\end{equation*}
Now we estimate the first term on the right-hand side of the last inequality. By Lemma \ref{lem-1-4}, it follows that
\begin{equation*}\label{eq3-15}
\begin{split}
\left(1+\frac{\mu\alpha(t)}{4}\right)\|X(t)\|^2 &=\left(1+\frac{\mu\alpha(t)}{4}\right)\left\|\bar{x}(t) - x_* - \alpha(t) \left(\nabla f(\bar{x}(t))-\nabla f(x_*)\right)\right\|^2 \\
&\leq \left(1+\frac{\mu\alpha(t)}{4}\right)\bigg(1-\frac{\mu\alpha(t)}{2} \bigg)^2\| \bar{x}(t)-x_*\|^2.
\end{split}
\end{equation*}
Note that since $\alpha (t) \leq \frac{2}{L+\mu}\leq \frac{2}{\mu}$, it follows that
\begin{equation*}
\begin{split}
\Big( 1+ \frac{\mu \alpha (t)}{4}\Big) \Big( 1- \mu \alpha (t) + \frac{\mu^2 \alpha (t)^2}{4}\Big)& = 1-\frac{3}{4} \mu \alpha (t) + \frac{\mu^3}{16}\alpha (t)^3
\\
& = 1- \frac{\mu \alpha (t) }{2} - \Big( \frac{\mu \alpha (t)}{4} - \frac{\mu^3 \alpha (t)^3}{16}\Big)
\\
&\leq 1- \frac{\mu \alpha (t)}{2}.
\end{split}
\end{equation*}
Using this result and the face that $\|\bx(t)-\x_*\|^2=n\|\bar{x}(t)-x_*\|^2$, the first term is bounded by
\begin{equation}\label{eq-A-8}
\begin{split}
n\left(1+\frac{\mu\alpha(t)}{4}\right)\|X(t)\|^2 &\leq \left(1- \frac{\mu \alpha (t)}{2}\right)n\| \bar{x}(t)-x_*\|^2\\
&= \left(1- \frac{\mu \alpha (t)}{2}\right)\| \bx(t)-\x_*\|^2
\end{split}
\end{equation}
Next, we estimate the second term. Note that by the Cauchy-Schwarz inequality and Lemma \ref{lem-1-5}, we get 
\begin{equation*}
\begin{split}
\left\| \nabla f(\bar{x}(t)) -\frac{1}{n}\sum^n_{i=1}\nabla f_i(x_i(t)) \right\|^2 &= \frac{1}{n^2}\left\|\sum^n_{i=1}\left(\nabla f_i(\bar{x}(t)) - \nabla f_i({x}_i(t)) \right) \right\|^2\\
&\leq \frac{1}{n}\sum^n_{i=1} \left\| \nabla f_i(\bar{x}(t)) - \nabla f_i({x}_i(t))\right\|^2\\
&\leq \frac{L^2}{n} \sum_{i=1}^n \|\bar{x}(t) - x_i (t)\|^2
\end{split}
\end{equation*}
Using this inequality, we have 
\begin{equation*}\label{eq3-16}
\begin{split}
\left(1+\frac{4}{\mu\alpha(t)}\right)\alpha(t)^2 \|Y(t)\|^2 &= \left(1+\frac{4}{\mu\alpha(t)}\right)\alpha(t)^2\left\| \nabla f(\bar{x}(t)) -\frac{1}{n}\sum^n_{i=1}\nabla f_i(x_i(t)) \right\|^2 \\
&\leq \left(1+\frac{4}{\mu\alpha(t)}\right)\frac{L^2\alpha (t)^2}{n} \sum_{i=1}^n \|\bar{x}(t) - x_i (t)\|^2.
\end{split}
\end{equation*}
Therefore, using the fact that $\|\x(t)-\bx(t)\|^2 = \sum^n_{i=1}\|x_i(t)-\bar{x}(t)\|^2$, the second term is bounded by
\begin{equation}\label{eq-A-9}
\begin{split}
n\left(1+\frac{4}{\mu\alpha(t)}\right)\alpha(t)^2 \|Y(t)\|^2 
&\leq \left(1+\frac{4}{\mu\alpha(t)}\right)L^2\alpha (t)^2 \sum_{i=1}^n \|x_i (t)-\bar{x}(t)\|^2 \\
&= \left(1+\frac{4}{\mu\alpha(t)}\right)L^2\alpha (t)^2\|\x(t)-\bx(t)\|^2.
\end{split}
\end{equation}
Combining \eqref{eq-A-8} and \eqref{eq-A-9}, it follows that 
\begin{equation*}
\begin{split}
 &n \left\|\bar{x}(t)-x_* -\frac{\alpha(t)}{n}\sum^n_{l=1} \left(\nabla f_l({x}_l(t)) -  \nabla f_l(x_*)\right) \right\|^2\\
 &\leq\bigg(1+\frac{\mu\alpha(t)}{4}\bigg)\bigg(1-\frac{\mu\alpha(t)}{2} \bigg)^2\|\bar{\mathbf{x}}(t) - \mathbf{x}_* \|^2 + \bigg(1+\frac{4}{\mu\alpha(t)}\bigg)L^2\alpha(t)^2 \|\mathbf{x}(t) - \bar{\mathbf{x}}(t) \|^2\\
 &\leq \left(1- \frac{\mu \alpha (t)}{2}\right) \|\bar{\mathbf{x}}(t) - \mathbf{x}_* \|^2 + \left(L^2\alpha(t)^2+\frac{4L^2\alpha(t)}{\mu}\right) \|\mathbf{x}(t) - \bar{\mathbf{x}}(t) \|^2,
\end{split}
\end{equation*}
which completes the proof.
\end{proof}

\begin{lem}\label{lem-1-1}
Suppose that the mixing matrix $W$ satisfies the Assumption \ref{ass-1-1}. Then, for any $x = (x_1, \cdots, x_n) \in \mathbb{R}^{d\cdot n}$ we have
\begin{equation*}
\sum_{i=1}^n \left\| \sum_{j=1}^n w_{ij} (x_j - \bar{x}) \right\|^2 \leq \beta^2 \sum_{i=1}^n \|x_i - \bar{x}\|^2,
\end{equation*}
where $\bar{x} = \frac{1}{n}\sum^{n}_{i=1} x_i$.
\end{lem}
\begin{proof}
We refer to Lemma 1 in \cite{PN}.
\end{proof} 

\subsection{Proof of Propositions \ref{prop-3-1} and \ref{prop-3-2}}\label{secA-3}
We are now prepared to prove   Propositions \ref{prop-3-1} and \ref{prop-3-2}. In this subsection, we used the following constants.
\begin{equation}\label{eq-A-12}
c_1 := 3L^2\bigg(1+\frac{1}{\delta}\bigg),\ c_2 := 3nD^2\bigg(1+\frac{1}{\delta}\bigg),\ c_3 := c_1 + L^2 , \ c_4 := \frac{4L^2}{\mu}.
\end{equation}
It is worth noting that the constants in Propositions \ref{prop-3-1} and \ref{prop-3-2}, denoted as $c_1,c_2,c_3,$ and $c_4$, correspond to the case where $\delta = \frac{1}{\beta}-1>0$ in \eqref{eq-A-12}. With this choice of $\delta$, we have the equality $(1+\delta)\beta^2 = \beta$.
\begin{proof}[Proof of Proposition \ref{prop-3-1}]
We recall Young's inequality, which is useful for our derivations: For any vectors $u$ and $v$ in $\mathbb{R}^d$, we have
\begin{equation*}
\|u+v\|^2 \leq (1+\delta) \|u\|^2 + \left( 1+ \frac{1}{\delta}\right) \|v\|^2
\end{equation*}
Using the DPG algorithm in \eqref{scheme}, we can write $\|x_i(t+1) - \bar{x}(t+1)\|^2$ as 
$$
\left\|\mathcal{P}_{\Omega}\bigg[\sum^n_{j=1}w_{ij}x_j(t) -\alpha(t)\nabla f_i({x}_i(t))\bigg] - \frac{1}{n}\sum^n_{k=1}\mathcal{P}_{\Omega}\bigg[\sum^n_{j=1}w_{kj}x_j(t) -\alpha(t)\nabla f_i({x}_k(t))\bigg]\right\|^2.
$$
Summing this $i=1$ to $n$ and applying Lemma \ref{lem-1-7}, it follows that 
\begin{equation}\label{eq-a-12}
\begin{split}
&\sum^n_{i=1} \|x_i(t+1) - \bar{x}(t+1)\|^2\\
&\leq \sum^{n}_{i=1}\left\| \sum^n_{j=1}w_{ij}\left(x_j(t) - \bar{x}(t)\right) -\alpha(t)\nabla f_i({x}_i(t)) + \frac{\alpha(t)}{n}\sum^n_{l=1}\nabla f_l({x}_l(t)) \right\|^2.
\end{split}
\end{equation}
where Applying Young's inequality with 
$$
u = \sum^n_{j=1}w_{ij}\left(x_j(t) - \bar{x}(t)\right),\ v=\alpha(t)\left(\nabla f_i({x}_i(t)) + \frac{1}{n}\sum^n_{l=1}\nabla f_l({x}_l(t))\right),
$$ the last term is bounded as
$$
(1+\delta)\left\|W(\mathbf{x}(t)-\bar{\mathbf{x}}(t))\right\|^2 + \sum^n_{i=1} \left(1+\frac{1}{\delta}\right)\alpha(t)^2 \left\| \nabla f_i({x}_i(t)) + \frac{1}{n}\sum^n_{l=1}\nabla f_l({x}_l(t))  \right\|^2.
$$
Here we used the fact that $\sum^n_{i=1}\|x_i(t)-\bar{x}(t)\|^2 =  \|\x(t) - \bx(t)\|^2$.  Now, applying Lemma \ref{lem-1-1}, we bound the first term as:
$$
(1+\delta)\beta^2\|\mathbf{x}(t)- \bar{\mathbf{x}}(t)\|^2,
$$
Using Lemma \ref{lem-1-5}, the second term is bounded as
$$
\left(1+\frac{1}{\delta}\right)\alpha(t)^2\left(3L^2 \| \mathbf{x}(t) - \bar{\mathbf{x}}(t)\|^2  + 3L^2 \| \bar{\mathbf{x}}(t)- \mathbf{x}_*\|^2 +3nD^2\right).
$$
Putting all together, we obtain the desired estimate:
\begin{equation*}
\begin{split}
&\sum^n_{i=1} \|x_i(t+1) - \bar{x}(t+1)\|^2 \\
\leq&(1+\delta)\left\|W(\mathbf{x}(t)-\bar{\mathbf{x}}(t))\right\|^2 + \sum^n_{i=1} \left(1+\frac{1}{\delta}\right)\alpha(t)^2 \left\| \nabla f_i({x}_i(t)) + \frac{1}{n}\sum^n_{l=1}\nabla f_l({x}_l(t))  \right\|^2\\
\leq &\left( 3L^2\left(1+\frac{1}{\delta}\right)\alpha(t)^2  + (1+\delta)\beta^2\right)\|\mathbf{x}(t)- \bar{\mathbf{x}}(t)\|^2  \\
&\qquad  \qquad\qquad \qquad  \qquad +3L^2\left(1+\frac{1}{\delta}\right)\alpha(t)^2 \| \bar{\mathbf{x}}(t)- \mathbf{x}_*\|^2 + 3nD^2\left(1+\frac{1}{\delta}\right)\alpha(t)^2\\
=&(c_1\alpha(t)^2 + (1+\delta)\beta^2)\| \mathbf{x}(t) - \bar{\mathbf{x}}(t)\|^2 + c_1\alpha(t)^2 \|\bar{\mathbf{x}}(t)-\mathbf{x}_*\|^2 + c_2\alpha(t)^2,
\end{split}
\end{equation*}
which completes the proof.
\end{proof}  
\begin{proof}[Proof of Proposition \ref{prop-3-2}]
By the DPG algorithm and using the contraction property of the projection operator in Lemma \ref{lem-1-2}, we deduce
\begin{equation*}
\begin{split}
\| x_i(t+1) - x_* \|^2&= \bigg\| \mathcal{P}_\Omega\bigg[\sum^n_{j=1}w_{ij} x_j(t) -\alpha(t)\nabla f_i({x}_i(t))\bigg] - \mathcal{P}_\Omega\left[ x_* -\alpha(t)\nabla f(x_*)\right]\bigg\|^2 \\
&\leq \bigg\|\sum^n_{j=1}w_{ij} x_j(t) -x_* -\alpha(t)\left(\nabla f_i({x}_i(t))-\nabla f(x_*)\right)\bigg\|^2.
\end{split}
\end{equation*}
For the first equality, we used the fact that $x_* = \mathcal{P}_\Omega\left[ x_* -\alpha(t)\nabla f(x_*)\right]$, which follows from $x_* = \arg\min_{x\in\Omega} f(x)$.
Summing up the above inequality from $i=1$ to $n$, we have
\begin{equation}\label{relation}
\sum^{n}_{i=1}\| x_i(t+1) - x_* \|^2  \leq \sum^{n}_{i=1}\bigg\|\sum^n_{j=1}w_{ij} x_j(t) -x_*-\alpha(t)\left(\nabla f_i({x}_i(t))-\nabla f(x_*)\right) \bigg\|^2.
\end{equation}
Denoting $P(t)$ as
$$
P(t) = \bar{x}(t) -\frac{\alpha(t)}{n}\sum^n_{l=1}\nabla f_l(x_l(t)),
$$
we find the following identity of the right-hand side of \eqref{relation}:
\begin{equation}\label{relation2}
\begin{split}
 &  \sum^{n}_{i=1}\bigg\|\sum^n_{j=1}w_{ij} x_j(t) -x_*-\alpha(t) \left(\nabla f_i({x}_i(t))-\nabla f(x_*)\right) \bigg\|^2\\
 &= \sum^{n}_{i=1} \bigg\|P(t)-(x_* -\alpha(t)  \nabla f(x_*))  -P(t) + \sum^n_{j=1}w_{ij}x_j(t)  -\alpha(t)\nabla f_i({x}_i(t)) \bigg\|^2\\
&= n\bigg\|P(t)-\left(x_* -\alpha(t)  \nabla f(x_*)\right)\bigg\|^2 + \sum^{n}_{i=1} \bigg\|\sum^n_{j=1}w_{ij}x_j(t)  -\alpha(t)\nabla f_i({x}_i(t)) - P(t) \bigg\|^2.
\end{split}
\end{equation} 
For the last equality, we used
\begin{equation*}
\begin{split}
&\sum^n_{i=1}\Big(\sum^n_{j=1}w_{ij}x_j(t)  -\alpha(t)\nabla f_i({x}_i(t)) - P(t)\Big) \\
&=\sum^n_{j=1}x_j(t) - \alpha(t)\sum^n_{i=1}\nabla f_i(x_i(t)) - n P(t)= 0,
\end{split}
\end{equation*}
and the fact that $\sum^n_{i=1}\|a + b_i\|^2 = n\|a\|^2 + \sum^n_{i=1}b_i$ if $\sum^n_{i=1}b_i =0$. 
Now we estimate the first term on the right-hand side of the last equality in \eqref{relation2} as follows:
\begin{equation*}
\begin{split}
&n\bigg\|P(t)-\left(x_* -\alpha(t)  \nabla f(x_*)\right)\bigg\|^2\\
&= n\left\|\bar{x}(t)-x_* -\alpha(t)\left(\frac{1}{n}\sum^n_{i=1} \nabla f_i({x}_i(t)) -  \nabla f(x_*)\right) \right\|^2\\
&\leq
\bigg(1- \frac{\mu \alpha (t)}{2}\bigg)\|\bar{\mathbf{x}}(t) - \mathbf{x}_* \|^2 + \bigg(L^2\alpha(t)^2+\frac{4L^2\alpha(t)}{\mu}\bigg) \|\bar{\mathbf{x}}(t) - \mathbf{x}(t)\|^2,
\end{split}
\end{equation*}
where we applied Lemma \ref{lem-3-8} for the last inequality. 

Now we estimate the second term on the right-hand side of the last equality in \eqref{relation2}, which is equal to the right-hand side of \eqref{eq-a-12}. Therefore, we can apply the estimate for \eqref{eq-a-12} to the second term in the last equality of \eqref{relation2} and obtain
\begin{equation*}
\begin{split}
&\sum^{n}_{i=1} \bigg\|\sum^n_{j=1}w_{ij}x_j(t)  -\alpha(t)\nabla f_i({x}_i(t)) - P(t) \bigg\|^2\\
&=\sum^{n}_{i=1}\bigg\|  \sum^n_{j=1}w_{ij}\Big(x_j(t) - \bar{x}(t)\Big) -\alpha(t)\Big(\nabla f_i({x}_i(t))- \frac{1}{n}\sum^n_{l=1}\nabla f_l(x_l(t))\Big)\bigg\|^2\\
&\leq (c_1 \alpha (t)^2 + (1+\delta)\beta^2)\| \mathbf{x}(t) - \bar{\mathbf{x}}(t)\|^2 +c_1 \alpha (t)^2 \|\bar{\mathbf{x}}(t)-\mathbf{x}_*\|^2 + c_2 \alpha(t)^2.
\end{split}
\end{equation*}
Putting the above two estimates in the last term of \eqref{relation2}, we get
\begin{equation*}
\begin{split}
\sum^{n}_{i=1}\| x_i(t+1) - x_* \|^2&= \|\mathbf{x}(t+1) -\bar{\mathbf{x}}(t+1)\|^2 + \|\bar{\mathbf{x}}(t+1) - \mathbf{x}_*\|^2.\\
& \leq\bigg(\Big( c_1 + L^2\Big) \alpha (t)^2 +\frac{4L^2\alpha(t)}{\mu} + \tilde{\beta}\bigg) \| \mathbf{x}(t) - \bar{\mathbf{x}}(t)\|^2
\\
&\qquad \qquad + \bigg(1- \frac{\mu \alpha (t)}{2} + c_1 \alpha(t)^2\bigg)\|\bar{\mathbf{x}}(t)-\mathbf{x}_*\|^2 + c_2 \alpha(t)^2,
\end{split}
\end{equation*}
which completes the proof.
\end{proof}
\bibliographystyle{siamplain}
\bibliography{references}

\begin{thebibliography}{10}

\bibitem{AGL}
{\sc M.~Akbari, B.~Gharesifard, and T.~Linder}, {\em Distributed online convex optimization on time-varying directed graphs}, IEEE Transactions on Control of Network Systems, 4 (2015), pp.~417--428, \url{https://doi.org/10.1109/TCNS.2015.2505149}.

\bibitem{ALBR}
{\sc M.~Assran, N.~Loizou, N.~Ballas, and M.~Rabbat}, {\em Stochastic gradient push for distributed deep learning}, in International Conference on Machine Learning, PMLR, 2019, pp.~344--353.

\bibitem{B1}
{\sc A.~Beck}, {\em First-order methods in optimization}, SIAM, 2017, \url{https://doi.org/10.1137/1.9781611974997}.

\bibitem{BCN}
{\sc L.~Bottou, F.~E. Curtis, and J.~Nocedal}, {\em Optimization methods for large-scale machine learning}, SIAM review, 60 (2018), pp.~223--311, \url{https://doi.org/10.1137/16M1080173}.

\bibitem{Boyd_Gossip}
{\sc S.~Boyd, A.~Ghosh, B.~Prabhakar, and D.~Shah}, {\em Randomized gossip algorithms}, IEEE transactions on information theory, 52 (2006), pp.~2508--2530, \url{https://doi.org/10.1109/TIT.2006.874516}.

\bibitem{B}
{\sc S.~Bubeck et~al.}, {\em Convex optimization: Algorithms and complexity}, Foundations and Trends{\textregistered} in Machine Learning, 8 (2015), pp.~231--357, \url{https://doi.org/10.1561/2200000050}.

\bibitem{BCM}
{\sc F.~Bullo, J.~Cort{\'e}s, and S.~Martinez}, {\em Distributed control of robotic networks: a mathematical approach to motion coordination algorithms}, vol.~27, Princeton University Press, 2009.

\bibitem{CB}
{\sc X.~Cao and T.~Ba{\c{s}}ar}, {\em Decentralized online convex optimization based on signs of relative states}, Automatica, 129 (2021), p.~109676, \url{https://doi.org/10.1016/j.automatica.2021.109676}.

\bibitem{CYRC}
{\sc Y.~Cao, W.~Yu, W.~Ren, and G.~Chen}, {\em An overview of recent progress in the study of distributed multi-agent coordination}, IEEE Transactions on Industrial informatics, 9 (2012), pp.~427--438, \url{https://doi.org/10.1109/TII.2012.2219061}.

\bibitem{I}
{\sc A.~I.-A. Chen}, {\em Fast distributed first-order methods}, PhD thesis, Massachusetts Institute of Technology, 2012.

\bibitem{CK}
{\sc W.~Choi and J.~Kim}, {\em On the convergence of decentralized gradient descent with diminishing stepsize, revisited},  (2022), \url{https://arxiv.org/abs/2203.09079}.

\bibitem{deng2012mnist}
{\sc L.~Deng}, {\em The mnist database of handwritten digit images for machine learning research}, IEEE Signal Processing Magazine, 29 (2012), pp.~141--142.

\bibitem{DMR}
{\sc T.~Doan, S.~Maguluri, and J.~Romberg}, {\em Finite-time analysis of distributed td (0) with linear function approximation on multi-agent reinforcement learning}, in International Conference on Machine Learning, PMLR, 2019, pp.~1626--1635.

\bibitem{DMDT}
{\sc Z.~Dong, S.~Mao, W.~Du, and Y.~Tang}, {\em Distributed constrained optimization with linear convergence rate}, in 2020 IEEE 16th International Conference on Control \& Automation (ICCA), IEEE, 2020, pp.~937--942.

\bibitem{FP}
{\sc F.~Facchinei and J.-S. Pang}, {\em Finite-dimensional variational inequalities and complementarity problems}, Springer, 2003, \url{https://doi.org/10.1007/b97543}.

\bibitem{FCG}
{\sc P.~A. Forero, A.~Cano, and G.~B. Giannakis}, {\em Consensus-based distributed linear support vector machines}, in Proceedings of the 9th ACM/IEEE International Conference on Information Processing in Sensor Networks, 2010, pp.~35--46.

\bibitem{HCM}
{\sc S.~Hosseini, A.~Chapman, and M.~Mesbahi}, {\em Online distributed convex optimization on dynamic networks}, IEEE Transactions on Automatic Control, 61 (2016), pp.~3545--3550, \url{https://doi.org/10.1109/TAC.2016.2525928}.

\bibitem{JLYC}
{\sc B.~Jin, H.~Li, W.~Yan, and M.~Cao}, {\em Distributed model predictive control and optimization for linear systems with global constraints and time-varying communication}, IEEE Transactions on Automatic Control, 66 (2020), pp.~3393--3400, \url{https://doi.org/10.1109/TAC.2020.3021528}.

\bibitem{KHT}
{\sc Y.~Kajiyama, N.~Hayashi, and S.~Takai}, {\em Distributed subgradient method with edge-based event-triggered communication}, IEEE Transactions on Automatic Control, 63 (2018), pp.~2248--2255, \url{https://doi.org/10.1109/TAC.2018.2800760}.

\bibitem{KLBJ}
{\sc A.~Koloskova, N.~Loizou, S.~Boreiri, M.~Jaggi, and S.~Stich}, {\em A unified theory of decentralized sgd with changing topology and local updates}, in International Conference on Machine Learning, PMLR, 2020, pp.~5381--5393.

\bibitem{LCF}
{\sc J.~Lei, H.-F. Chen, and H.-T. Fang}, {\em Primal--dual algorithm for distributed constrained optimization}, Systems \& Control Letters, 96 (2016), pp.~110--117.

\bibitem{QT}
{\sc Q.~Ling and Z.~Tian}, {\em Decentralized sparse signal recovery for compressive sleeping wireless sensor networks}, IEEE Transactions on Signal Processing, 58 (2010), pp.~3816--3827, \url{https://doi.org/10.1109/TSP.2010.2047721}.

\bibitem{LLSX}
{\sc C.~Liu, H.~Li, Y.~Shi, and D.~Xu}, {\em Distributed event-triggered gradient method for constrained convex minimization}, IEEE Transactions on Automatic Control, 65 (2019), pp.~778--785, \url{https://doi.org/10.1109/TAC.2019.2916985}.

\bibitem{LQX}
{\sc S.~Liu, Z.~Qiu, and L.~Xie}, {\em Convergence rate analysis of distributed optimization with projected subgradient algorithm}, Automatica, 83 (2017), pp.~162--169, \url{https://doi.org/10.1016/j.automatica.2017.06.011}.

\bibitem{NO3}
{\sc A.~Nedi{\'c} and A.~Olshevsky}, {\em Stochastic gradient-push for strongly convex functions on time-varying directed graphs}, IEEE Transactions on Automatic Control, 61 (2016), pp.~3936--3947, \url{https://doi.org/10.1109/TAC.2016.2529285}.

\bibitem{nedic2017}
{\sc A.~Nedic, A.~Olshevsky, and W.~Shi}, {\em Achieving geometric convergence for distributed optimization over time-varying graphs}, SIAM Journal on Optimization, 27 (2017), pp.~2597--2633.

\bibitem{NO1}
{\sc A.~Nedic and A.~Ozdaglar}, {\em Distributed subgradient methods for multi-agent optimization}, IEEE Transactions on Automatic Control, 54 (2009), pp.~48--61, \url{https://doi.org/10.1109/TAC.2008.2009515}.

\bibitem{PN}
{\sc S.~Pu and A.~Nedi{\'c}}, {\em Distributed stochastic gradient tracking methods}, Mathematical Programming, 187 (2021), pp.~409--457, \url{https://doi.org/10.1007/s10107-020-01487-0}.

\bibitem{GN}
{\sc G.~Qu and N.~Li}, {\em Harnessing smoothness to accelerate distributed optimization}, IEEE Transactions on Control of Network Systems, 5 (2017), pp.~1245--1260, \url{https://doi.org/10.1109/TCNS.2017.2698261}.

\bibitem{RB}
{\sc H.~Raja and W.~U. Bajwa}, {\em Cloud k-svd: A collaborative dictionary learning algorithm for big, distributed data}, IEEE Transactions on Signal Processing, 64 (2015), pp.~173--188, \url{https://doi.org/10.1109/TSP.2015.2472372}.

\bibitem{RG}
{\sc A.~Rogozin and A.~Gasnikov}, {\em Projected gradient method for decentralized optimization over time-varying networks},  (2019), \url{https://arxiv.org/abs/1911.08527}.

\bibitem{S}
{\sc A.~H. Sayed}, {\em Diffusion adaptation over networks}, in Academic Press Library in Signal Processing, vol.~3, Elsevier, 2014, pp.~323--453.

\bibitem{SBP}
{\sc F.~Shahriari-Mehr, D.~Bosch, and A.~Panahi}, {\em Decentralized constrained optimization: Double averaging and gradient projection}, in 2021 60th IEEE Conference on Decision and Control (CDC), IEEE, 2021, pp.~2400--2406.

\bibitem{RNV}
{\sc S.~Sundhar~Ram, A.~Nedi{\'c}, and V.~V. Veeravalli}, {\em Distributed stochastic subgradient projection algorithms for convex optimization}, Journal of optimization theory and applications, 147 (2010), pp.~516--545, \url{https://doi.org/doi.org/10.1007/s10957-010-9737-7}.

\bibitem{WLSM}
{\sc H.-T. Wai, J.~Lafond, A.~Scaglione, and E.~Moulines}, {\em Decentralized frank--wolfe algorithm for convex and nonconvex problems}, IEEE Transactions on Automatic Control, 62 (2017), pp.~5522--5537.

\bibitem{WS}
{\sc D.~J. Watts and S.~H. Strogatz}, {\em Collective dynamics of ‘small-world’networks}, nature, 393 (1998), pp.~440--442.

\bibitem{XB}
{\sc L.~Xiao and S.~Boyd}, {\em Fast linear iterations for distributed averaging}, Systems \& Control Letters, 53 (2004), pp.~65--78.

\bibitem{YLY}
{\sc K.~Yuan, Q.~Ling, and W.~Yin}, {\em On the convergence of decentralized gradient descent}, SIAM Journal on Optimization, 26 (2016), pp.~1835--1854, \url{https://doi.org/10.1137/130943170}.

\end{thebibliography}

\end{document}